\documentclass[leqno,a4paper]{article}

\usepackage[english]{babel}
\usepackage{amsmath}
\usepackage{amsthm}
\usepackage{amssymb}
\usepackage{enumerate}
\usepackage{xspace}
\usepackage{euscript}
\usepackage{graphicx}
\usepackage{amscd}
\usepackage{epsfig}
\usepackage{tabularx}
        \newtheorem{lemma}{Lemma}[section]
        \newtheorem{proposition}[lemma]{Proposition}
        \newtheorem{theorem}[lemma]{Theorem}
        \newtheorem{definition}{Definition}[section]
        \newtheorem{remark}[lemma]{Remark}

\title{A parabolic inverse problem with mixed boundary data.
Stability estimates for the unknown boundary and impedance}

\author{V. Bacchelli\footnotemark[1],\hspace{1em}
M. Di Cristo\footnotemark[2],\hspace{1em}
E. Sincich\footnotemark[3],\hspace{1em}
S. Vessella\footnotemark[4]}

\begin{document}

\date{}

\maketitle
\footnotetext[1]{Politecnico di Milano, e-mail: \texttt{valeria.bacchelli@polimi.it}}
\footnotetext[2]{Politecnico di Milano, e-mail: \texttt{michele.dicristo@polimi.it}}
\footnotetext[3]{Universit\`a di Trieste, e-mail: \texttt{esincich@units.it}}
\footnotetext[4]{Universit\`a di Firenze, e-mail: \texttt{sergio.vessella@dmd.unifi.it}}

\setcounter{section}{0}
\setcounter{secnumdepth}{2}

\begin{abstract}
We consider the problem of determining an unaccessible part of the boundary of a conductor
by mean of thermal measurements. We study a problem of corrosion where a Robin type condition
is prescribed on the damaged part and we prove logarithmic stability estimate.
\end{abstract}

\section{Introduction}
In this paper we consider the problem of determine an inaccessible portion $I$ of the boundary of a conductor
body $\Omega\subset\mathbb R^n$ by mean of thermal measurements, performed on an accessible part $A$ of its boundary.
In particular we analyze the situation in which there might be a corrosion occurring on $I$ and our aim is
to recover information on this damaged part that can not be directly inspected.
This leads to a parabolic equation with a Robin type condition on the inaccessible part of $\partial\Omega$ and
a Dirichlet or Neumann condition, according whether we prescribe a temperature or a heat flux, on $A$
(see \cite{Bi-Ce-Fa-In, Bi-Fa-In, In}).
This kinds of boundary conditions are known as mixed type.

Assuming $\partial\Omega=\overline A\cup \overline I$
and $\mathrm{Int}_{\partial\Omega}(A)\cap\mathrm{Int}_{\partial\Omega}(I)=\emptyset$
and denoting by $\gamma(x,t)$ the surface impedance on $I$,
we prescribe a heat flux $g$ on $A$ that induces a temperature $u$ in $\Omega$ solution to
\begin{equation}
\label{eq-tot}
\left\{\begin{array}{ll}
u_t=\Delta u & \textrm{in }\Omega\times[0,T],\\[2mm]
u(x,0)=0 & \textrm{in }\Omega,\\[2mm]
\displaystyle{\frac{\partial u}{\partial\nu}(x,t)}=g(x,t) & \textrm{on }A\times[0,T],\\[2mm]
\displaystyle{\frac{\partial u}{\partial\nu}}+\gamma u=0 & \textrm{on }I\times[0,T],
\end{array}\right.
\end{equation}
where $\nu$ is the outer unit normal to $\partial\Omega$.

The inverse problem we are addressing to is to recover information on the unknown part $I$
of the boundary and on the impedance coefficient
when thermal measurements of the form $\{g, u_{|\Sigma}\}$,
where $\Sigma\subset A,$ are available.
Particularly we are interested in the issue of stability, that is we want to study the continuous dependence
of the solution $I$ from the boundary data.

This problem has been considered in the stationary case in \cite{Al-Be-Ro-Ve} with a Dirichlet or Neumann condition
on the unknown boundary $I$. The authors show that, keeping as minimal as possible the a priori assumptions
on the unknowns, the solution depends continuously on the boundary measurements with a rate of continuity of
logarithmic type, which is the best possible as shown in \cite{DC-Ro}.
The non stationary analysis has been carried on in \cite{Ca-Ro-Ve}
(see also \cite{Ve} for a first study of this problem).
Also here stability estimates of logarithmic type are provided.
A refined analysis of the problem has been proposed in \cite{DC-Ro-Ve}.
Beside the more general framework considered by the authors, precisely they deal with thermal conductivity
depending on time and space, this article contains a detailed study on the optimality of logarithmic rate of continuity
in the parabolic case.

The problem of determine part of the boundary with Robin type condition has been considered in \cite{Ba} where
a uniqueness result in the stationary case is proved, provided two measurements are performed.
In view of the example given in \cite{Ca-Kr}, the number of measurements turns out to be optimal.
Stability estimates of logarithmic type has been obtained in \cite{Si}.
This result is optimal as well (see \cite{DC-Ro}).
Let us finally mention \cite{Pa-Pi},
where a uniqueness result under weaker regularity assumptions on the boundary has been obtained.

In this paper we show that $I$ depends on thermal boundary measurements with a rate of continuity
of logarithmic type. As in the elliptic case, we perform two boundary measurements, precisely we prescribe two
different heat fluxes and we read the corresponding temperatures on a portion of the accessible part
of the boundary.
The optimality of two measurements is still an unsolved question. We also believe that the argument used in
\cite{DC-Ro-Ve} to prove exponential instability could be applied in the present setting through minor adaptations.

Main ideas and tools can be outlined as follows.
\begin{itemize}
\item[i)] Evaluating how much the error on measurements can effect the error on an auxiliary function $\lambda$
obtained as the ratio of the solutions $\tilde u$ and $u$ corresponding to the heat fluxes $\tilde g$
and $g$. Such a control has been obtained by combining two arguments.
The first relies on smallness propagation estimates based on an iterated use of two--sphere and one--cylinder
inequality (\cite{Es-Fe-Ve, Ve2}). The latter is a lower bound for the solution $u$ achieved by
combining the Harnack inequality up to the inaccessible boundary (see Proposition \ref{harnack-lemma})
with an iterated application of the interior Harnack inequality (\cite{Mo, Mo2}).

\item[ii)] A lower bound for $\lambda$ which has been established by the use of quantitative estimates
of unique continuation and a proper choice of the given heat fluxes. Precisely we prescribe functions $g$
and $\tilde g$ that are linearly independent with a quantitative control of such an independence.

\item[iii)] Using i) and ii) we prove a first rough estimate of log--log type for the Hausdorff distance
between the unknown domains.
Employing, then, in a more refined way the above mention estimates, in particular using the two--sphere
one--cylinder inequality at the boundary (see Theorem \ref{tre-teo} and \cite{Es-Fe-Ve, Ve2}),
and a geometric argument, we get the logarithmic estimate. The stability for the unknown impedance
follows from stability estimates for the underlying Cauchy problem and the stability result
for the unknown boundary.
\end{itemize}

For the sake of exposition, we have chosen to study the inverse problem with the constant coefficients
equation. All proofs, though, can be simply adapted to the equation with coefficients depending
on time and space with reasonable assumptions on them.
Indeed, we deal with an auxiliary function $\lambda$ that solves an equation with variable coefficients.

The plan of the paper is the following. The main result is stated in Section \ref{sec2},
where we also give notations and definitions. In Section \ref{sec3} we provide a proof of this result based
on some auxiliary propositions proved in the subsequent Section \ref{sec4}.

\section{Main Result}
\label{sec2}
We begin by giving some notations and definitions.
For every $x\in\mathbb R^n$, $n\geq2$, $x=(x_1,\dots,x_n)$, we set $x=(x',x_n)$,
where $x'\in\mathbb R^{n-1}$ and $x_n\in\mathbb R$.
We denote by $B_r(x)$ and $B_r'(x')$ respectively the open ball in $\mathbb R^n$
centered at $x$ of radius $r$ and the open ball in $\mathbb R^{n-1}$ centered at $x'$ of radius $r$.
Sometimes we shall write $B_r$ and $B'_r$ instead of $B_r(0)$ and $B'_r(0)$ respectively.
For given numbers $r,t>0$, and a function $\varphi$ defined on $\mathbb R^{n-1}\times \mathbb R$, we define
$\Omega_r=\{x\in\Omega\,:\,\mathrm{dist}(x,\partial\Omega)>r\}$,
$Q_{r,\varphi}^t=\{(x,s)\in B_r\times(t-r^2,t)\,:\,x_n>\varphi(x')\}$
and $\Gamma_{r,\varphi}^t=\{(x,s)\in B'_r\times(t-r^2,t)\,:\,x_n=\varphi(x')\}$.

Denoting by $D$ a open subset of $\mathbb R^{n+1}$, for a function $f$ defined on $D$ and $\alpha\in(0,1]$, we define
$$[f]_{\alpha;D}=\sup\left\{\frac{|f(x,t)-f(y,s)|}{(|x-y|^2+|t-s|)^{\alpha/2}}\,:\,(x,t),(y,s)\in D, (x,t)\neq(y,s)\right\}.$$
If $\alpha\in(0,2]$, we set
$$<f>_{\alpha;D}=\sup\left\{\frac{|f(x,t)-f(y,s)|}{|t-s|^{\alpha/2}}\,:\,(x,t),(y,s)\in D, t\neq s \right\}.$$
Let $k$ be a positive integer, $D$ an open subset of $\mathbb R^{n+1}$, $f$ a sufficiently smooth function
and $\alpha\in(0,1]$. We denote by
\begin{eqnarray*}
&&[f]_{k+\alpha;D}=\sum_{|\beta|+2j=k}[\partial_x^\beta\partial_t^j f]_{\alpha;D},\\[2mm]
&&<f>_{k+\alpha;D}=\sum_{|\beta|+2j=k-1}<\partial_x^\beta\partial_t^j f>_{1+\alpha;D},
\end{eqnarray*}
where for a multi-index $\beta=(\beta_1,\dots,\beta_n)$, $\beta_i\in\mathbb N\cup\{0\}$, $i=1,\dots,n$, we
have used the notation
$\partial_x^\beta\partial_t^j f=\frac{\partial^{|\beta|+k}f}{\partial_{x_1}^{\beta_1}
\dots \partial_{x_n}^{\beta_n}\partial_{t}^{k}}$, with $|\beta|=\sum_{i=1}^{n}\beta_i$.
If $\alpha\in(0,1]$ and $[f]_{\alpha;D}$ is finite, we shall say that $f$ belongs to $C^{0,\alpha}(D)$.
Let $k$ be a positive integer, $\alpha\in(0,1]$ and $D$ an open subset of $\mathbb R^{n+1}$,
we shall say that $f$ belongs to the class $C^{k,\alpha}(D)$ whenever for every non-negative integer $j$
such that $|\beta|+2j\leq k$, there exist the derivatives $\partial_{x}^{\beta}\partial_{t}^{j}f$
and the quantities $\sup_D |\partial_{x}^{\beta}\partial_{t}^{j}f|$, $[f]_{k+\alpha;D}$ and
$<f>_{k+\alpha;D}$ are finite.
If $f$ is a function not depending on $t$, we keep the definition above by considering a function $\tilde f$
defined on $\Omega\times\mathbb R$, $\Omega\subset\mathbb R^n$, such that $\tilde f(x,t)=f(x)$
for every $(x,t)\in\Omega\times\mathbb R$
and we shall say that $f\in C^{k,\alpha}(\Omega)$ whenever $\tilde f\in C^{k,\alpha}(\Omega\times\mathbb R)$.
Throughout the paper we will make use of standard  Sobolev spaces.
We refer the reader to \cite{Li-Ma} for details.

\begin{definition}
Let $\Omega$ be a domain in $\mathbb R^n$. Given $\alpha$, $\alpha\in(0,1]$ and $k$, $k\in\mathbb N$,
we say that $\partial\Omega$
is of class $C^{k,\alpha}$ with constants $r_0,L$ if for any $P\in\partial\Omega$ there exists a rigid transformation
of $\mathbb R^n$ under which we have $P\equiv0$ and
$$\Omega\cap B_{r_0}=\{x\in B_{r_0}\,:\,x_n>\varphi(x')\},$$
where $\varphi$ is a $C^{k,\alpha}$ function on $B'_{r_0}$ satisfying the following condition
$\varphi(0)=|\nabla_{x'}\varphi(0)|=0$ and $\|\varphi\|_{C^{k,\alpha}(B'_{r_0})}\leq Lr_0$.
\end{definition}
\begin{remark}
We have chosen to normalize all norms in such a way that their terms are dimensional homogeneous and
coincide with the standard definition when $r_0=1$. For instance, for any
$\varphi\in C^{k,\alpha}(B'_{r_0}\times(-r^2_0,r^2_0))$ we set
\begin{eqnarray*}
&&\|\varphi\|_{C^{k,\alpha}(B'_{r_0}\times(-r^2_0,r^2_0))}=
\sum_{l=0}^{k} r_0^l\sum_{|\beta|+2j=l}\|\partial_x^\beta\partial_t^j\varphi\|_{L^\infty(B'_{r_0}\times(-r^2_0,r^2_0))}\\
&&\qquad\qquad +r_0^{k+\alpha}\left(<\varphi>_{k+\alpha;B'_{r_0}\times(-r^2_0,r^2_0)}
+[\varphi]_{k+\alpha;B'_{r_0}\times(-r^2_0,r^2_0)}\right).
\end{eqnarray*}
Similarly we set
$$\|u\|_{L^2(D)}=r_0^{-\frac{n+2}{2}}\left(\int_D u^2dxdt\right)^{1/2},$$
where $D$ is a domain in $\mathbb R^{n+1}$.
\end{remark}
We shall use letters $C,C_0,C_1,\dots$ to denote constants. The value of these constants may
change from line to line and their dependance will specified everywhere they appear.

\noindent\textbf{Assumptions on the domain}.
Let $r_0,M,L$ be given positive numbers. We assume that $\Omega$ is a bounded domain in $\mathbb R^n$ such that
\begin{subequations}
\label{Omega}
\begin{equation}
\label{Omega1}
|\Omega|\leq Mr_0^n,
\end{equation}
where $|\Omega|$ denotes the Lebesgue measure,
\begin{equation}
\label{Omega2}
\partial\Omega=\overline A\cup \overline I,
\end{equation}
where $A$ and $I$ are open subsets of $\partial\Omega$,
$A\cap I=\emptyset$, and
\begin{equation}
\label{Omega3}
\partial\Omega \textrm{ is of class } C^{1,1} \textrm{ with constants } r_0,L.
\end{equation}
Also, we denote by $\Sigma$ an open portion of $\partial\Omega$ so that there exists a point $P_0\in\Sigma$
such that
\begin{equation}
\label{Omega4}
\partial\Omega\cap B_{r_0}(P_0)\subset\Sigma.
\end{equation}
\end{subequations}
\textbf{Assumptions on the boundary data}. Given positive constants $E,\Phi_0,\Phi_1$,
on the accessible part $A$ of the boundary of $\Omega$
we shall prescribe two different heat fluxes $g$ and $\tilde g$ such that
\begin{subequations}
\label{phi}
\begin{equation}
\label{phi0}
g,\tilde g\in C^{0,1}(A\times[0,T]),
\end{equation}
\begin{equation}
\label{phi00}
\mathrm{supp}\,g,\,\,\mathrm{supp}\,\tilde g\subset A^{r_0}\times[0,T],
\end{equation}
where $A^{r_0}=\{x\in A\,:\,\mathrm{dist}(x,I)>r_0\}$ ,
\begin{equation}
\label{phi0001}
A^{2r_0}\neq\emptyset,
\end{equation}
\begin{equation}
\label{phi001}
\|g\|_{C^{0,1}(A\times[0,T])},\,\|\tilde g\|_{C^{0,1}(A\times[0,T])}\leq E,
\end{equation}
\begin{equation}
\label{phib}
\exists\, t_{1}>0\text{ such that }g(x,t)=\widetilde{g}(x,t),\quad\text{ for }0\leq
t\leq t_{1},\,x\in A,
\end{equation}
\begin{equation}
\label{phia}
\left\|\frac{\tilde g}{g}-\left(\frac{\tilde g}{g}\right)_{A\times[0,T]}\right\|_{L^2(A\times[0,T])}\geq\Phi_0>0,
\end{equation}
where
$$\left(\frac{\tilde g}{g}\right)_{A\times[0,T]}=\frac{1}{|A|T}\int_{A\times[0,T]}\frac{\tilde g}{g}d\sigma dt,$$
and
\begin{equation}
\label{039}
g(x,t)\geq \Phi _{1} r_0^{-1}>0,\text{ in }A^{2r_0}\times [0,T].
\end{equation}
\end{subequations}

\noindent\textbf{Assumptions on the surface impedance}.
Given a positive number $\overline\gamma$,
the surface impedance $\gamma$ of the unknown boundary $I$ is such that
\begin{subequations}
\label{gamma}
\begin{equation}
\label{gamma0}
\gamma\in C^{0,1}(I\times[0,T]),
\qquad\textrm{with}\qquad
\textrm{supp}\gamma \subset I\times[0,T]
\end{equation}
and
\begin{equation}
\label{040}
0\leq\gamma(x,t)\leq\overline{\gamma}.
\end{equation}
\end{subequations}

\begin{remark}
\label{regol}
By \cite[Theorem 6.46, page 141]{Li}, if \eqref{Omega1}, \eqref{Omega3}, \eqref{phi0}, \eqref{phi00},
\eqref{gamma0} are fulfilled,
there exists a unique solution $u$ of the problem \eqref{eq-tot},
$u\in C^{1,\alpha}(\overline\Omega\times[0,T])$ such that
\begin{equation}
\label{reg-lib}
\|u\|_{C^{1,\alpha}(\overline\Omega\times[0,T])}\leq C_0\|g\|_{C^{0,1}(A\times[0,T])},
\quad\forall\alpha\in(0,1),
\end{equation}
where $C_0$ is a positive constant depending $n,\Omega,\overline\gamma$.
\end{remark}
\noindent From now on we shall fix an $\alpha$, $\alpha\in[1/2,1)$.

\noindent We denote by $\Omega_i$, $i=1,2$, two bounded domains on $\mathbb R^n$ satisfying \eqref{Omega} such that
$$\partial\Omega_i=\overline A\cup \overline I_i,\qquad
\textrm{Int}_{\partial\Omega_i}(A)\cap \textrm{Int}_{\partial\Omega_i}(I_i)=\emptyset,
\qquad i=1,2,$$
where the accessible part $A$ of the boundary is the same for both sets
and by $\gamma_i(x,t)$, $i=1,2$, the boundary impedance on $I_i$, $i=1,2$, respectively satisfying \eqref{gamma}.
Let also $t_2$ and $t_3$ such that $0<t_1<t_2<t_3\leq T$.

In the sequel we shall refer to numbers $M,L,\Phi_0,\Phi_1,\overline\gamma,r^2_0/T,r^2_0/t_1$,
as the a priori data.

\begin{theorem}
\label{main-theor}
For $i=1,2$, let $u_i\in C^{1,\alpha}(\overline\Omega\times[0,T])$ be the solution to \eqref{eq-tot} when
$\Omega=\Omega_i$, $\gamma=\gamma_i$ and let $\tilde u_i\in C^{1,\alpha}(\overline\Omega\times[0,T])$
be the solution to \eqref{eq-tot}
when $\Omega=\Omega_i$, $\gamma=\gamma_i,$ $g=\tilde g$.
Let \eqref{Omega1}--\eqref{040} be satisfied.
Assume for $\varepsilon>0$
\begin{equation}
\label{eps}
\begin{array}{l}
\|u_1-u_2\|_{L^2(\Sigma\times[t_1,T])}\leq\varepsilon,\\[2mm]
\|\tilde u_1-\tilde u_2\|_{L^2(\Sigma\times[t_1,T])}\leq\varepsilon,
\end{array}
\end{equation}
where $\Sigma\subset A$ is an open subset,
then
\begin{equation}
\label{tesi}
d_\mathcal H(\overline\Omega_1,\overline\Omega_2)\leq r_0\eta(\varepsilon),
\end{equation}
where $\eta$ is a continuous increasing function on $[0,+\infty)$ satisfying
\begin{equation}
\label{eta}
\eta(s)\leq C|\log s|^{-\beta},
\end{equation}
for every $0<s<1$, with $C>0$ and $\beta$ depending on the a priori data only.
Furthermore, for any $t\in(t_1,T)$,
\begin{eqnarray}
\label{gasta}
{\sup_{\begin{array}{c}{\scriptstyle{P\in {I_1}^{r_0}}}\\
\scriptstyle{Q \in B_{2\eta(\varepsilon)}(P)\cap {I_2}^{r_0}}
\end{array}}
|\gamma_2(Q,t)-\gamma_1(P,t)|}\leq \eta(\varepsilon)\ ,
\end{eqnarray}
where $\eta$ is defined as in \eqref{eta} up to a possible replace of constants $C$ and $\beta$.
\end{theorem}
Here $d_\mathcal H$ stands for the Hausdorff distance.

\section{Proof of Theorem \ref{main-theor}}
\label{sec3}

We first observe that the solution $u_{1}\in C^{1,\alpha}(\overline{\Omega}_1\times [0,T])$
of problem \eqref{eq-tot} with boundary data $g$ satisfying \eqref{039}, is such that
\begin{equation}
u_{1}(x,t)>0,\text{ in }\Omega _{1}\times (0,T].  \label{positive}
\end{equation}
Namely, by contradiction, if
\begin{equation*}
\underset{\overline{\Omega}_{1}\times [0,T]}\min u_{1}(x,t)=u_{1}(\overline{x},\overline{t})\leq 0,
\end{equation*}
then, by maximum principle, the point $(\overline{x},\overline{t})$ belongs
to the parabolic boundary.

If $(\overline{x},\overline{t})\in A\times (0,T],$ by Hopf lemma we would
have $\dfrac{\partial u_{1}}{\partial \nu }(\overline{x},\overline{t})<0,$
that contradicts (\ref{039}). If $(\overline{x},\overline{t})\in I\times
(0,T],$ again by Hopf lemma we would have $\dfrac{\partial u_{1}}{\partial
\nu }(\overline{x},\overline{t})<0,$ which contradicts the Robin condition
that $u_{1}$ satisfies on $I$ because%
\begin{equation*}
\frac{\partial u_{1}}{\partial \nu }(\overline{x},\overline{t})
=-\gamma u_{1}(\overline{x},\overline{t})\geq 0.
\end{equation*}
Then $(\overline{x},\overline{t})=(\overline{x},0)$,
\begin{equation*}
\underset{\overline{\Omega }_{1}\times [0,T]}\min u_{1}(x,t)=u_{1}(\overline{x},0)=0
\end{equation*}
and we get \eqref{positive}.

\noindent The same is true for $u_{2}(x,t).$

\noindent By \eqref{positive} we can define, for $i=1,2$,
\begin{equation}
\lambda_i(x,t)=\frac{\widetilde{u}_{i}(x,t)}{u_{i}(x,t)}-1,
\text{ in }\overline{\Omega}_{i}\times [t_{1},T].  \label{52}
\end{equation}
By straightforward calculation we notice that $\lambda _{i}(x,t)$ satisfies
the problem
\begin{equation}
\label{eq-lambda}
\left\{\begin{array}{ll}
\partial_t\lambda_i(x,t)=\mathrm{div}(u_i^2\nabla \lambda_i(x,t)) & \textrm{in }\Omega_i\times[t_1,T],\\[2mm]
\lambda_i(x,t_1)=0 & \textrm{in }\Omega_i,\\[2mm]
u_i^2\displaystyle{\frac{\partial\lambda_i}{\partial\nu}(x,t)
=u_i\tilde g(x,t)-\tilde u_i g(x,t)} & \textrm{on }A\times[t_1,T],\\[2mm]
\displaystyle{u_i^2\frac{\partial \lambda_i}{\partial\nu}}(x,t)=0 & \textrm{on }I_i\times[t_1,T].
\end{array}\right.
\end{equation}
By standard estimates of solutions of parabolic problem \cite{Li},
by (\ref{phi00}), (\ref{phi001}), (\ref{reg-lib}), we have
\begin{equation}
\left\Vert \lambda_{i}\right\Vert _{C^{1,\alpha}(\Omega _{i}\times(t_{1},T))}\leq C,
\label{54}
\end{equation}
where $\alpha\in(0,1)$ and $C$ depends on the a priori data only.

With the change of variable \eqref{52} we can deal with the new problem \eqref{eq-lambda},
where we have a homogeneous Neumann condition on $I$.

In the next propositions, whose proofs are postponed to Section \ref{sec4},
we provide stability estimates of unique continuation from Cauchy data when \eqref{eps} holds true,
then a lower bound on $u$, where $u$ is solution to \eqref{eq-tot} and a lower bound
of the integral  of $\lambda_i$ in term of the boundary data.

The proof of Theorem \ref{main-theor} will be obtained from the following sequence of propositions.

We shall denote by $G$ the connected component of $\Omega_1\cap\Omega_2$ such that $A\subset\overline G$.

\begin{proposition}[Stability estimates of unique continuation from Cauchy data]
\label{cauchy-data}
Let hypothesis of Theorem \ref{main-theor} be satisfied.
Then there exists a positive constant $C$ depending on the a priori data only such that for $i=1,2$
\begin{equation}
\label{sta-lam}
\int_{\Omega_i\setminus G}u^2_i(x,t)\lambda_i^2(x,t)dx\leq Cr_0^n\eta_1(\varepsilon)\qquad \forall\,t\in[t_1,T],
\end{equation}
where $\eta_1$ is an increasing continuous function on $[0,+\infty)$ which satisfies
$$\eta_1(s)\leq(\log|\log s|)^{-\beta_1},$$
for every $0<s<1$, with $\beta_1>0$.
\end{proposition}

\begin{proposition}[Improved stability estimates]
\label{cauchy-data2}
Let hypothesis of Proposition \ref{cauchy-data} be fulfilled.
In addition, assume there exist constants $L>0$ and $r_1$, $0<r_1<r_0$ such that
$\partial G$ is of Lipschitz class with constant $r_1$ and $L$.
Then there exists a positive constant $C$ depending on the a priori data only such that
\begin{equation}
\label{sta-lam2}
\int_{\Omega_i\setminus G} u_i^2(x,t)\lambda_i^2(x,t)dx\leq Cr_0^n\eta_2(\varepsilon)\qquad \forall\,t\in[t_1,T],
\end{equation}
where $\eta_2$ is an increasing continuous function on $[0,+\infty)$ which satisfies
$$\eta_2(s)\leq|\log s|^{-\beta_2},$$
for every $0<s<1$, with $\beta_2>0$.
\end{proposition}

\begin{proposition}[Lower bound on $u$]
\label{lemma-lower}
Let $u\in C^{1,\alpha}(\overline\Omega\times[0,T])$ be a solution to \eqref{eq-tot} with boundary data $g$
satisfying \eqref{phi0}, \eqref{phi00}, \eqref{phi001}, \eqref{039}.
Then there exists a positive constant $c_0$, $0<c_0<1$, depending on the a priori data
except  $\Phi_0,\Phi_1$ such that
\begin{equation}
\label{(1)-1}
u(x,t)\geq c_0\Phi_1, \qquad \textrm{for }t\geq t_1, x\in\Omega,
\end{equation}
where $t_1$ as in \eqref{039}.
\end{proposition}

\begin{proposition}[Lower bound for $\lambda$]
\label{stimadalbasso}
For every $\rho>0$ and for every $x_0\in \Omega_{\rho}$, we have for $i=1,2$,
\begin{eqnarray}
\label{0.18}
\int_{t_1}^{T}\int_{B_{\rho}(x_0)}\lambda_i^2(x,t) dxdt \geq C_\rho r_0^{n+2}\Phi_0,
\end{eqnarray}
where $C_\rho>0$ is a constant depending on the a priori data and $\rho$ only.
\end{proposition}

To better deal with the Hausdorff distance, we introduce a variation of it that, though it
is not a metric, we shall call modified distance (see also \cite{Al-Be-Ro-Ve,DC-Ro-Ve}).
\begin{definition}
We call modified distance between $\Omega _{1}$ and $\Omega _{2}$ the
number
\begin{equation}
\label{731}
d_{m}( \Omega _{1},\Omega _{2}) =\max \left\{
\sup_{x\in
\partial \Omega _{1}}\mathrm{dist}( x,\overline{\Omega }_{2}),
\sup_{x\in \partial \Omega _{2}}\mathrm{dist}(x,\overline{\Omega }_{1}) \right\}.
\end{equation}
\end{definition}

Note that
\begin{equation}
\label{741}
d_{m}( \Omega _{1},\Omega _{2}) \leq d_{\mathcal{H}}(\overline{\Omega }_{1},\overline{\Omega }_{2}),
\end{equation}
but, in general, the reverse inequality does not hold.
However we have the following result, \cite{Al-Be-Ro-Ve}

\begin{proposition}[Proposition 3.6 \cite{Al-Be-Ro-Ve}]
\label{moddistprop}
Let $\Omega _{1}$ and $\Omega _{2}$ be bounded domains satisfying \eqref{Omega}.
There exist numbers $d_{0}>0$, $\tilde r\in (0,r_{0}]$, such that
$\frac{d_{0}}{r_{0}}$ and $\frac{\tilde r}{r_{0}}$ depend on $E$ only
and the following facts hold true. If
\begin{equation}
\label{751}
d_{\mathcal{H}}( \overline{\Omega }_{1},\overline{\Omega}_{2}) \leq d_{0},
\end{equation}
then there exists an absolute constant $C>0$ such that
\begin{equation}
\label{761}
d_{\mathcal{H}}( \overline{\Omega }_{1},\overline{\Omega }_{2})
\leq Cd_{m}( \Omega _{1},\Omega _{2}),
\end{equation}
and any
connected component of $\Omega _{1}\cap \Omega_{2}$ has boundary
of Lipschitz class with constants $\tilde r$, $L_1$ where $\tilde r$ is as
above and $L_1>0$ depends on $L$ only.
\end{proposition}
Last tool we need for the proof of Theorem \ref{main-theor} is related with quantitative form
of unique continuation property.

\begin{theorem}[two--sphere one--cylinder inequality at the boundary]
\label{tre-teo}
Let $\lambda,\Lambda$ and $R$ be positive numbers, with $\lambda\in(0,1]$ and $t_0\in\mathbb R$.
Let $L$ be a parabolic operator $L=\partial_i(g^{ij}(x,t)\partial_j)-\partial_t$,
where $\{g^{ij}(x,t)\}_{i,j=1}^n$ is a real symmetric $n\times n$ matrix.
For $\xi\in\mathbb R^n$, $(x,t),(y,\tau)\in\mathbb R^{n+1}$, assume that
$$\lambda|\xi|^2\leq\sum_{i,j=1}^n g^{ij}(x,t)\xi_i\xi_j\leq\lambda^{-1}|\xi|^2$$
and
$$\left(\sum_{i,j=1}^n(g^{ij}(x,t)-g^{ij}(y,\tau))^2\right)^{1/2}\leq\frac{\Lambda}{R}
\left(|x-y|^2+|t-\tau|\right)^{1/2}.$$
Let $u\in H^{2,1}(Q_{R,\varphi}^{t_0})$ be such that
$$|Lu|\leq\Lambda\left(\frac{|\nabla u|}{R}+\frac{|u|}{R^2}\right),\qquad\textrm{in }Q_{R,\varphi}^{t_0}$$
and
$$g^{ij}\frac{\partial u}{\partial x_j}(x,t)\nu_i=0,\qquad\forall\,(x,t)\in\Gamma_{R,\varphi}^{t_0}.$$
There exist constants $s_1\in(0,1)$ and $C$, $C>0$, depending on $\lambda,\Lambda$ and $L$ only
such that for every $r$, $0<r\leq\rho\leq s_1R$ we have
\begin{equation}
\label{tre}
\int_{B_\rho\cap\Omega}u^2(x,t_0)dx\leq
\frac{CR^2}{\rho^2}\left(R^{-2}\int_{Q_{R,\varphi}^{t_0}}u^2dxdt\right)^{1-\theta}
\left(\int_{B_r\cap\Omega}u^2(x,t_0)dx\right)^\theta,
\end{equation}
where $\theta=\frac{1}{C\log\frac{R}{r}}$.
\end{theorem}
\begin{proof}
The proof can be obtained along the line of \cite[Theorem 3.3.5]{Ve2} through slight modifications
due to the different boundary condition we have on $\Gamma_{R,\varphi}^{t_0}$ (see also \cite{Es-Fe-Ve}
where a similar problem is studied).
\end{proof}
An inequality similar to \eqref{tre} can be obtained for cylinder and spheres entirely contained
in the domain $\Omega$. We refer the interested reader to \cite[Theorem 3.3.3]{Ve2}.

\begin{proof}[Proof of Theorem \ref{main-theor}]
For the sake of brevity we denote $d=d_\mathcal H(\overline\Omega_1,\overline\Omega_2)$ and
$d_m=d_m(\overline\Omega_1,\overline\Omega_2)$.
The proof follows the lines of the proof of \cite[Theorem 4.1]{Ca-Ro-Ve}.
We shall sketch only the main items, using
Propositions \ref{cauchy-data}, \ref{cauchy-data2}, \ref{lemma-lower}, \ref{stimadalbasso}.
Let us prove that if $\eta>0$ is such that
\begin{equation}
\int_{t_{1}}^{T}\int_{\Omega _{1}\backslash G}u_1^2(x,t)\lambda _{1}^{2}(x,t)dxdt\leq\eta,
\qquad\forall t\in \lbrack t_{1},T],
\label{115}
\end{equation}
then there exists a constant $C$, depending on the a priori data only,
such that
\begin{equation}
d_{m}\leq C\eta^{K},
\label{120}
\end{equation}
where $K$ depends on the a priori data only.
We may assume, without loss of generality, that there exists $x_0\in I_1\subset\partial\Omega_1$ such that
$$\mathrm{dist}(x_0,\overline\Omega_2)=d_m.$$
We apply now the two--sphere one-cylinder inequality at the boundary \cite[Theorem 3.3.5]{Ve2} with $r=d_m$,
$\rho=cr_0$ and $R=r_0$, we integrate over the time interval $[0,T]$ and we get
\begin{eqnarray}
\label{1-1}
\quad\int_0^{T}\int_{Q_{\rho,\varphi}(t)}\lambda_1^2(x,t)dxdt
&\leq&\frac{CR^2}{\rho^2}
\left(R^{-2}\int_{Q_{R,\varphi}^t}\lambda_1^2(x,t)dxdt\right)^{1-\theta}\\[2mm]
&&\times\left(\int_0^T\int_{{Q_{r,\varphi}}(t)}\lambda_1^2(x,t)dxdt\right)^{\theta},\nonumber
\end{eqnarray}
where $\theta=\frac{1}{C\log\left(\frac{R}{r}\right)}$.
Recalling \eqref{sta-lam2} and \eqref{0.18}, we get the following inequality
\begin{equation}
\label{1-2}
\Phi_0\leq CA^{1-\theta}\eta^\theta,
\qquad\textrm{where }
A=\int_{\Omega_1\times[0,T]}u_1^2(x,t)dxdt.
\end{equation}
Developing \eqref{1-2} we arrive to
$$\rho_1\leq C r_0\left(\frac{\eta}{A}\right)^{\left|\log\frac{a}{A}\right|},$$
which leads to
\begin{equation}
\label{1-3}
d_m\leq C_1\eta^{C_0},
\end{equation}
where $C_0=|\log \Phi_0/A|$ and $C_1$ is a positive constant depending on the a priori data only.
Let us consider now the Hausdorff distance $d$.
With no loss of generality, we may assume there exists $y_0\in\overline\Omega_1$ such that
$\mathrm{dist}(y_0,\overline\Omega_2)=d$. Denote by $\delta=\mathrm{dist}(y_0,\partial\Omega_1)$.
We distinguish three cases.

i) $\delta\leq d/2$.

\noindent We take $z_0\in\partial\Omega_1$ such that $|y_0-z_0|=\delta$ and we have
$$d_m\geq \mathrm{dist}(z_0,\overline\Omega_2)\geq d-\delta\geq\delta/2,$$
hence $\delta\leq 2d_m$, that is \eqref{1-3} holds for $d$ as well.

ii) $d/2<\delta\leq d_0/2$.

\noindent This implies $d<d_0$ and by Proposition \ref{moddistprop} we obtain \eqref{1-3} for $d$.

iii) $\delta>\max\left\{d/2,d_0/2\right\}$.

\noindent We observe that if $d_0/2<d/2$ then \eqref{1-3} for $d$ is trivial.
On the other hand $d/2<d_0/2$ implies $\delta\geq d/2$.
Let us denote $d_1=\min\left\{\frac{d}{2},\frac{s_1r_0}{2}\right\}$,
where $s_1\in(0,1)$ has been introduced in Theorem \ref{tre-teo} and it depends on the
a priori data only.
We have $B_{d_1}(y_0)\subset\Omega_1\setminus\Omega_2$ and $B_{s_1r_0/2}(y_0)\subset\Omega_1$
because $\delta>\max\{d/2,d_0/2\}\geq d_0/2>s_1r_0/2$.
Applying again Theorem \ref{tre-teo} with $\rho_1=d_1$, $\rho_2=s_1r_0/2$, $R=d_0$,
$T_1=T/2$, $\tau=T/4$ and proceeding as in \eqref{1-1}
and applying Proposition \ref{cauchy-data}, \ref{cauchy-data2},  we get the thesis \eqref{tesi}.

Let us prove now \eqref{gasta}.

First we observe that, in general, the Hausdorff distances
$d_{H}(\overline\Omega_1, \overline\Omega_2)$ and
$d_{H}(\partial{\Omega_1}, \partial{\Omega_2})$ are not equivalent.
However, in our regularity assumptions, the following estimate
\begin{eqnarray}
\label{e1}
d_{H}(\partial{\Omega_1}, \partial{\Omega_2})\le \eta(\varepsilon)
\end{eqnarray}
can be derived from \eqref{tesi}
using the arguments contained in the proof of Proposition 3.6 in \cite{Al-Be-Ro-Ve}.
We consider a point $P\in {I_1}^{r_0}$, a point $Q\in B_{2\eta(\varepsilon)}(P)\cap {I_2}^{r_0}$ and $t\in (t_1,T)$.
With no loss of generality we may assume that $P,Q \in \overline\Omega_1$,
hence we have that for any $t\in (t_1,T)$,
\begin{eqnarray}
\label{e2}
|\gamma_2(Q,t)-\gamma_1(P,t)|
&\leq&\left| \frac{\partial u_1}{\partial \nu}(P,t)\frac{1}{u_1(P,t)}
-\frac{\partial u_1}{\partial\nu}(Q,t)\frac{1}{u_1(Q,t)} \right|\nonumber\\
&&+\left| \frac{\partial u_2}{\partial \nu}(Q,t)\frac{1}{u_2(Q,t)}
-\frac{\partial u_1}{\partial \nu}(Q,t)\frac{1}{u_1(Q,t)}\right|.
\end{eqnarray}
We can split the first term on the right hand side of \eqref{e2} as follows
\begin{eqnarray*}
\left| \frac{\partial u_1}{\partial \nu}(P,t)\frac{1}{u_1(P,t)}
-\frac{\partial u_1}{\partial \nu}(Q,t)\frac{1}{u_1(Q,t)} \right|
&\leq& \left| \frac{\partial u_1}{\partial \nu}(P,t)\frac{1}{u_1(P,t)}
-\frac{\partial u_1}{\partial \nu}(Q,t)\frac{1}{u_1(P,t)} \right|\nonumber \\
&&+ \left| \frac{\partial u_1}{\partial \nu}(Q,t)\frac{1}{u_1(P,t)}
-\frac{\partial u_1}{\partial \nu}(Q,t)\frac{1}{u_1(Q,t)} \right|.
\end{eqnarray*}
From Remark \ref{reg-lib} and Proposition \ref{lemma-lower} we can infer that
\begin{eqnarray}
\left| \frac{\partial u_1}{\partial \nu}(P,t)\frac{1}{u_1(P,t)}
-\frac{\partial u_1}{\partial \nu}(Q,t)\frac{1}{u_1(Q,t)} \right|\leq C|P-Q|.
\end{eqnarray}
where $C>0$ is a constant depending on the a priori data only.
Hence by \eqref{e1} we can infer that
\begin{eqnarray}
\label{e3}
\left| \frac{\partial u_1}{\partial \nu}(P,t)\frac{1}{u_1(P,t)}-
\frac{\partial u_1}{\partial \nu}(Q,t)\frac{1}{u_1(Q,t)} \right|\le \eta(\varepsilon)\ ,
\end{eqnarray}
up to a possible replacing of the constants $C$ and $\beta$ in \eqref{eta}.

Analogously we can split the second term on the right hand side of \eqref{e2} as follows
\begin{eqnarray*}
\left| \frac{\partial u_2}{\partial \nu}(Q,t)\frac{1}{u_2(Q,t)}-
\frac{\partial u_1}{\partial \nu}(Q,t)\frac{1}{u_1(Q,t)} \right|
&\leq& \left| \frac{\partial u_2}{\partial \nu}(Q,t)\frac{1}{u_2(Q,t)}
-\frac{\partial u_1}{\partial \nu}(Q,t)\frac{1}{u_2(Q,t)} \right|\nonumber \\
&&+ \left| \frac{\partial u_1}{\partial \nu}(Q,t)\frac{1}{u_2(Q,t)}
-\frac{\partial u_1}{\partial \nu}(Q,t)\frac{1}{u_1(Q,t)} \right|
\end{eqnarray*}
From Remark \ref{reg-lib} and Proposition \ref{lemma-lower} we can infer that
\begin{eqnarray*}
&&\left| \frac{\partial u_2}{\partial \nu}(Q,t)\frac{1}{u_2(Q,t)}
-\frac{\partial u_1}{\partial \nu}(Q,t)\frac{1}{u_1(Q,t)} \right|\\
&\leq& C\left|\frac{\partial u_2}{\partial \nu}(Q,t)
-\frac{\partial u_1}{\partial \nu}(Q,t) \right| + C\left|u_1(Q,t)-u_2(Q,t) \right|\
\end{eqnarray*}
where $C>0$ is a constant depending on the a priori data only.
Dealing as in Proposition \ref{sta-lam2}, we have that for any $t\in (t_1,T)$
\begin{eqnarray*}
\|u_1(\cdot,t)-u_2(\cdot,t) \|_{C^1({I_2}^{r_0})}\le \eta(\varepsilon).
\end{eqnarray*}
Hence we have that
\begin{eqnarray}
\label{e4}
\left| \frac{\partial u_2}{\partial \nu}(Q,t)\frac{1}{u_2(Q,t)}
-\frac{\partial u_1}{\partial \nu}(Q,t)\frac{1}{u_1(Q,t)} \right|\le \eta(\varepsilon).
\end{eqnarray}
Up to a possible replacing of the constants $C$ and $\beta$ in \eqref{eta}.
Combining \eqref{e3} and \eqref{e4} we obtain that for any $t\in (t_1,T)$
\begin{eqnarray}
|\gamma_2(Q,t)-\gamma_1(P,t)|\le \eta(\varepsilon).
\end{eqnarray}
Being such an estimate independent from $P$, $Q$ and $t$ the thesis \eqref{gasta} follows.
\end{proof}

\section{Proofs of Propositions \ref{cauchy-data}, \ref{cauchy-data2}, \ref{lemma-lower}, \ref{stimadalbasso}}
\label{sec4}
We proceed with the proof of Proposition \ref{cauchy-data}.
For this purpose we recall a result of \cite{Ca-Ro-Ve}, that will be used
several times in the next proofs.
\begin{theorem}[Theorem 3.3.1 \cite{Ca-Ro-Ve}]
\label{teo3.3.1}
Let $\Omega$ be a domain satisfying (\ref{Omega3}).
Let $P_{1}\in \Sigma$ be such that $\partial\Omega\cap B_{r_{0}}(P_{1})\subset\Sigma$.
Let $u\in H^{2,1}(\Omega \times (0,T))$ be a solution to
\begin{equation*}
\left\{\begin{array}{ll}
u_{t}=\Delta u & \textrm{ in }\Omega\times(0,T) \\
u(x,0)=0 &
\end{array}\right.
\end{equation*}
satisfying
\begin{equation*}
\left\Vert u\right\Vert_{H^{3/2,3/4}(\Sigma \times (0,T))}\leq C\,\overline{\varepsilon },
\qquad
\left\Vert\frac{\partial u}{\partial\nu}\right\Vert _{H^{1/2,1/4}(\Sigma\times (0,T))}\leq C\,\overline{\varepsilon },
\end{equation*}
where $C$ depends on $T,r_{0}$. Then, for every $t_{0}\in [0,T],$ we have
\begin{equation*}
\left\Vert u(\cdot,t_{0})\right\Vert_{L^{2}(B_{\overline{\theta}r_{0}}(P_{2}))}
\leq C\left\Vert u\right\Vert _{H^{2,1}(\Omega \times(0,T))}^{1-\overline\tau}\overline{\varepsilon }^{\overline\tau},
\end{equation*}
where $P_{2}=P_{1}-\overline{\theta}r_{0}\nu $, $\nu$ is the outer unit normal to $\Omega$ at $P_{1}$,
$\overline\tau$, $0<\overline\tau<1$, is an absolute constant, $\overline{\theta }$,
$0<\overline{\theta }<1/2$, depends on $L$ only, $C\geq 1$ depends on $L$ and $r_{0}^2/T$ only.
\end{theorem}

\begin{proof}[Proof of Proposition \ref{cauchy-data}]
We prove the proposition for $i=1$, as the other case $i=2$ is analogous.
In \cite{Li2} it is proved that
there exists a function $\widetilde{d}(x),$ labeled regularized distance,
$\widetilde{d}\in C^{2}(\Omega_{1})\cap C^{1,1}(\overline{\Omega}_{1})$
such that the following facts hold
\begin{equation*}
\begin{array}{l}
i)\,\xi_{1}\leq \displaystyle{\frac{\mathrm{dist}(x,\partial \Omega _{1})}{\widetilde{d}(x)}}\leq \xi_{2}, \\[2mm]
ii)\,\vert \nabla \widetilde{d}(x)\vert \geq c_{1}\text{ for every }
x\text{ such that }\textrm{dist}(x,\partial \Omega _{1})\leq br_{0}, \\[2mm]
iii)\,\Vert\widetilde{d}\Vert _{C^{1,1}}\leq c_{2}r_{0},
\end{array}
\end{equation*}
where $\xi_{1},\xi_{2},c_{1},c_{2},b$ are positive
constants depending on $L$ only (see also \cite[Lemma 5.2]{Al-Be-Ro-Ve}).
For $r>0$ we define
\begin{equation*}
\widetilde{\Omega }_{1,r}=\{x\in \Omega _{1}:\widetilde{d}(x)>r\}.
\end{equation*}
By \cite[Lemma 5.3]{Al-Be-Ro-Ve}, there exists a constant $a$, depending on $L$ only, such
that for every $r$, $0<r\leq ar_{0},$ $\widetilde{\Omega }_{1,r}$ is
connected with boundary of class $C^{1}$ and the following facts hold%
\begin{eqnarray}
\label{70}
&&\xi_{1}r\leq \textrm{dist}(x,\partial \Omega _{1})\leq \xi_{2}r,
\quad\forall\,x\in \partial \widetilde{\Omega }_{1,r},\\[2mm]
\label{71}
&&\left\vert \Omega _{1}\backslash \widetilde{\Omega }_{1,r}\right\vert \leq\xi_{3}Mr_{0}^{n-1}r,\\[2mm]
\label{72}
&&\left\vert \partial \widetilde{\Omega }_{1,r}\right\vert _{n-1}\leq \xi_{4}Mr_{0}^{n-1}.
\end{eqnarray}
Also, for every $x\in \partial\tilde{\Omega}_{1,r},$ there exists
$y\in \partial\Omega _{1}$ such that
\begin{equation}
\left\vert y-x\right\vert =\mathrm{dist}(x,\partial \Omega _{1}),\qquad
\left\vert \nu (x)-\nu (y)\right\vert \leq \xi_{5}\frac{r}{r_{0}},
\label{73}
\end{equation}
where $\nu (x),$ $\nu (y)$ denote the outer unit normal to $\widetilde{\Omega }_{1,r}$
at $x$ and to $\Omega _{1}$ at $y$ respectively.
Here $\xi_{j}$, $j=1,\dots,5,$ are constants depending on $L$ only.

Since $\xi_{2}r_{0}\theta \leq\frac{r_{0}}{16}$,
let us define $\theta =\min\{a,\frac{1}{16(1+M^{2})\xi_{2}}\}$ and
$\Sigma _{\xi_{2}r_{0}\theta }=\{x\in \Omega _{1}
\,:\,\textrm{dist}(x,\Sigma)=\xi_{2}r_{0}\theta\}\equiv \{x\in \Omega _{2}
\,:\,\textrm{dist}(x,\Sigma )=\xi_{2}r_{0}\theta \}$.

\noindent Let $\tilde{V}_{r}$ be the connected
component of $\tilde{\Omega }_{1,r}\cap \tilde{\Omega }_{2,r}$ whose
closure contains $\Sigma_{\xi_{2}r_{0}\theta}$.
We have
\begin{eqnarray}
&&\Omega _{1}\setminus G\subset[(\Omega _{1}\setminus\tilde{\Omega}_{1,r})\setminus G]
\cup[\tilde{\Omega }_{1,r}\backslash\tilde{V}_{r}]\\[2mm]
\label{80}
&&\partial (\tilde{\Omega}_{1,r}\setminus\tilde{V}_{r})=\tilde{\Gamma}_{1,r}\cup\tilde{\Gamma}_{2,r},
\end{eqnarray}
where $\tilde{\Gamma}_{1,r}$ is the part of the boundary contained in
$\partial\tilde{\Omega}_{1,r}$ and $\tilde{\Gamma }_{2,r}$ is
contained in $\partial \tilde{\Omega }_{2,r}\cap\partial\tilde{V}_{r}$.
We denote $\omega_{r}=\tilde{\Omega }_{1,r}\setminus\tilde{V}_{r}$.
For $t_{1}\leq t\leq T$, by (\ref{80})
\begin{eqnarray}
\label{81}
&&\int_{\Omega _{1}\setminus G}u_{1}^{2}(x,t)\lambda _{1}^{2}(x,t)dx\\[2mm]
&\leq&
\int_{(\Omega _{1}\setminus \tilde{\Omega }_{1,r})\setminus G}u_{1}^{2}(x,t)\lambda _{1}^{2}(x,t)dx+
\int_{\tilde{\Omega}_{1,r}\setminus \tilde{V}_{r}}u_{1}^{2}(x,t)\lambda _{1}^{2}(x,t)dx.\nonumber
\end{eqnarray}
Since $u_{1}^{2}(x,t)\lambda_{1}^{2}(x,t)=\left(\tilde{u}_1(x,t)-u_1(x,t)\right) ^{2}$,
by (\ref{36}), (\ref{42}), (\ref{71}$)$, there exists a constant $C$ depending on the a priori data only
such that, for $t_1\leq t\leq T$,
\begin{equation}
\label{82}
\int_{(\Omega _{1}\setminus\tilde{\Omega }_{1,r})\setminus G}u_{1}^{2}(x,t)\lambda_{1}^{2}(x,t)dx\leq Cr.
\end{equation}
Let us evaluate $\int_{\tilde{\Omega }_{1,r}\setminus\tilde{V}_{r}}u_{1}^{2}(x,t)\lambda _{1}^{2}(x,t)dx$.
Recalling that $\lambda _{1}$ solves (\ref{eq-lambda}), we get
\begin{equation}
\int_{t_{1}}^{t}ds\int_{\omega _{r}}u_{1}^{2}\lambda _{1t}\lambda_{1}dx
=\int_{t_{1}}^{t}ds\int_{\partial \omega _{r}}u_{1}^{2}\frac{\partial\lambda _{1}}{\partial \nu }\lambda _{1}d\sigma
-\int_{t_{1}}^{t}ds\int_{\omega_{r}}u_{1}^{2}\left\vert \nabla \lambda _{1}\right\vert ^{2}dx,
\label{83}
\end{equation}
where $\nu $ is the outer normal to $\omega_{r}$. We have, integrating by
parts the left hand side and since $\lambda_{1}(x,t_{1})=0$,
\begin{eqnarray}
\label{V4.43}
&&\frac{1}{2}\int_{t_{1}}^{t}ds\int_{\omega _{r}}u_{1}^{2}\lambda _{1t}\lambda
_{1}dx=\frac{1}{2}\int_{\omega _{r}}u_{1}^{2}\lambda _{1}^{2}\left\vert
_{t_{1}}^{t}\right. dx-\int_{t_{1}}^{t}ds\int_{\omega _{r}}\lambda
_{1}^{2}u_{1}u_{1t}\\[2mm]\nonumber
&=&\frac{1}{2}\int_{\omega _{r}}u_{1}^{2}(x,t)\lambda
_{1}^{2}(x,t)dx-\int_{t_{1}}^{t}ds\int_{\omega _{r}}\lambda
_{1}^{2}u_{1}u_{1t}.
\end{eqnarray}
Therefore plugging \eqref{83} into \eqref{V4.43}, we have
\begin{eqnarray*}
\frac{1}{2}\int_{\omega _{r}}u_{1}^{2}(x,t)\lambda_{1}^{2}(x,t)dx
+\int_{t_{1}}^{t}ds\int_{\omega _{r}}u_{1}^{2}\left\vert\nabla \lambda _{1}\right\vert ^{2}dx= \\
\int_{t_{1}}^{t}ds\int_{\omega _{r}}\lambda_{1}^{2}u_{1}u_{1t}+\int_{t_{1}}^{t}ds
\int_{\partial \omega _{r}}u_{1}^{2}\dfrac{\partial \lambda _{1}}{\partial \nu }\lambda _{1}d\sigma.
%\label{86}
\end{eqnarray*}
By Proposition \ref{lemma-lower}, (\ref{phi001}), (\ref{reg-lib}), (\ref{54}), we get
\begin{eqnarray}
&&C_1\int_{\omega _{r}}\lambda _{1}^{2}(x,t)d\sigma
\leq\int_{t_{1}}^{t}ds\int_{\omega _{r}}\lambda_{1}^{2}u_{1}u_{1t}
+\int_{t_{1}}^{t}ds\int_{\partial \omega _{r}}u_{1}^{2}\frac{\partial \lambda _{1}}{\partial \nu }\lambda _{1}dx\nonumber\\
&\leq&C_2\left(\int_{t_1}^{t_2}ds\int_{\omega_r}\lambda_1^2d\sigma+
T\max_{t_{1}\leq\theta\leq t}\int_{\partial\omega _{r}}
\left\vert\frac{\partial \lambda_{1}(x,\theta)}{\partial \nu }\right\vert d\sigma\right),
\label{87}
\end{eqnarray}
for $t_{1}\leq t\leq T$, where $C_1,C_2$ depend on a priori data only.
Denoting by $\overline t$, $\overline{t}\in [t_{1},t]$, where the maximum is achieved,
by Gronwall inequality, we obtain
\begin{equation}
\label{89bis}
\int_{\omega _{r}}\lambda _{1}^{2}(x,t)dx
\leq C\int_{\partial \omega_{r}}
\left\vert\frac{\partial\lambda_{1}(x,\overline{t})}{\partial \nu }\right\vert d\sigma,
\quad\forall\,t\in[t_{1},T],%\,\overline{t}\in [t_{1},t]
\end{equation}
where $C$ is a constant depending on the a priori data and on $t_{1}$.
Recalling that $\partial \omega_{r}\subset\tilde{\Gamma }_{1,r}\cup\tilde{\Gamma}_{2,r}$
we distinguish two situations:
either $x\in\tilde{\Gamma}_{1,r}$ or $x\in \tilde{\Gamma}_{2,r}$.
Let $x\in \tilde{\Gamma }_{1,r}$. By (\ref{70})
$\mathrm{dist}(x,\partial\Omega_{1})\leq\xi_{2}r$.
On the other hand $\mathrm{dist}(x,\Sigma )>\xi_{2}r$ (see \cite[Proposition 3.1]{Ve2}).
Hence there exists $y\in \partial \Omega_{1}\setminus\Sigma$ such that
$\left\vert y-x\right\vert=\mathrm{dist}(x,\partial \Omega_{1})\leq\xi_{2}r$.
Since $\displaystyle\frac{\partial\lambda _{1}}{\partial\nu }=0$ on
$I_{1}\times[t_{1},T]$, by \eqref{54}, (\ref{70}), (\ref{73})
\begin{equation}
\left\vert \frac{\partial \lambda _{1}(x,\overline{t})}{\partial \nu }
\right\vert \leq C\frac{r}{r^2_{0}},  \label{90}
\end{equation}
that implies
\begin{equation}
\int_{\tilde{\Gamma}_{1,r}}
\left\vert \frac{\partial \lambda _{1}(x,\overline{t})}{\partial\nu}\right\vert d\sigma
\leq C r\, r_{0}^{n-2},
\label{93}
\end{equation}
where $C$ is a constant depending on the a priori data only.
Let us consider now $x\in \tilde{\Gamma }_{2,r}$.
As before there exists $y\in \partial \Omega _{1}\setminus\Sigma $ such that
$\left\vert y-x\right\vert=\mathrm{dist}(x,\partial \Omega_{2})\leq\xi_{2}r$.
Since $\displaystyle\frac{\partial \lambda_{2}}{\partial\nu}=0$ on
$I_{2}\times[t_{1},T],$ we have that
\begin{equation*}
\left\vert\frac{\partial \lambda_{1}(x,\overline{t})}{\partial\nu}\right\vert
\leq \left\vert\nabla\lambda _{1}(x,\overline{t})-\nabla\lambda_{2}(x,\overline{t})\right\vert
+\left\vert \nabla \lambda _{2}(x,\overline{t})-\nabla \lambda_{2}(y,\overline{t})\right\vert.
\end{equation*}
Thus we get
\begin{equation}
\int_{\tilde{\Gamma}_{2,r}}\left\vert\frac{\partial\lambda _{1}(x,\overline{t})}{\partial\nu}\right\vert d\sigma
\leq C\,r_0^{n-2}r+\int_{\tilde{\Gamma}_{2,r}}\left\vert\nabla\lambda _{1}(x,\overline{t})
-\nabla \lambda _{2}(x,\overline{t})\right\vert d\sigma,
\label{95}
\end{equation}
where $C$ is a constant depending on the a priori data only.

Let us consider the integral on the right hand side of (\ref{95}). First
observe that
\begin{eqnarray}
\label{96}
&&\nabla \lambda _{1}-\nabla \lambda_{2}\\
&=&\frac{1}{\left(u_{1}u_{2}\right)^{2}}\{u_{1}u_{2}^{2}(\nabla \widetilde{u}_{1}-\nabla \widetilde{u}_{2})
-\tilde{u}_{1}u_{2}^{2}(\nabla u_{1}-\nabla u_{2})\nonumber\\
&&+u_{2}^{2}\nabla \tilde{u}_{2}\left( u_{1}-u_{2}\right)
-u_{2}^{2}\nabla u_{2}\left( \tilde{u}_{1}-\tilde{u}_{2}\right)
+(u_{2}^{2}-u_{1}^{2})(u_{2}\nabla \tilde{u}_{2}-\tilde{u}_{2}\nabla u_{2})\}.\nonumber
\end{eqnarray}
Now labeling $w=u_{1}-u_{2}$ and $\tilde{w}=\tilde{u}_{1}-\tilde{u}_{2}$
and taking into account (\ref{95}) and (\ref{96}), we get
\begin{eqnarray*}
&&\int_{\tilde{\Gamma}_{2,r}}\left\vert\frac{\partial \lambda _{1}(x,\overline{t})}{\partial\nu}\right\vert d\sigma\\[2mm]
&\leq&Cr_0^{n-1}\left(\frac{r}{r_0}
+r_0\max_{x\in \tilde{V}_{r}}\left\vert \nabla \tilde{w}(x,\overline{t})\right\vert
+r_0\max_{x\in \tilde{V}_{r}}\left\vert \nabla w(x,\overline{t})\right\vert\right.\\[2mm]
&&\left.\qquad+\max_{x\in \tilde{V}_{r}}\left\vert \tilde w(x,\overline{t})\right\vert
+\max_{x\in \tilde{V}_{r}}\left\vert w(x,\overline{t})\right\vert\right),
\end{eqnarray*}
where $C$ is a constant depending on the a priori data only.
To evaluate maximum of $w,\tilde w$ and their gradients we can proceed as \cite[Proposition 5.3]{Ca-Ro-Ve}.
Let us briefly sketch the main items of this proof for the $\nabla w$, for instance.
Suppose $\underset{x\in \overline{\tilde{V}}_{r}}\max\left\vert\nabla w(x,\overline{t})\right\vert
=\left\vert\nabla w(\overline{x},\overline{t})\right\vert
=\left\Vert \nabla w(\cdot,\overline t)\right\Vert_{L^{\infty}(B_{r}(\overline{x}))}$,
with $\overline{x}\in \tilde{V}_{r}$.
By interpolation inequality (see \cite[A2 page 553 with $\alpha=1/2$]{Ca-Ro-Ve}), \eqref{phi001} and \eqref{reg-lib}, we have
\begin{equation}
\left\Vert \nabla w(\cdot,\overline{t})\right\Vert_{L^{\infty }(B_{r}(\overline{x})}
\leq \frac{Cr_0}{r^{1+\frac{n}{3(n+2)}}}\,
\left(\int_{B_{r}(\overline{x})}w^2(x,\overline{t})dx\right)^{\frac{1}{3(n+2)}},
\label{106}
\end{equation}
where $C$ depends on $E$.
Now, in order to apply \cite[Theorem 3.3.1]{Ca-Ro-Ve}, we estimate
$\left\Vert w\right\Vert _{H^{3/2,1/2}(\Sigma \times (t_{1},T))}
=\left\Vert u_{1}-u_{2}\right\Vert _{H^{3/2,1/2}(\Sigma \times (t_{1},T))}$ in terms of
$\left\Vert u_{1}-u_{2}\right\Vert _{L^{2}(\Sigma \times (t_{1},T))}$ and of
the a priori data. The functions $w,$ $w_{t},$ $w_{tt}$ satisfy the problem
$$\left\{\begin{array}{ll}
z_{t}-\Delta z=0,&\textrm{ in }G\times (0,T),\\[2mm]
z=0,&\textrm{ on }G\times\{0\},\\[2mm]
\frac{\partial z}{\partial \nu }=0,&\textrm{ on }A\times (0,T).
\end{array}\right.$$
Hence, recalling that $A^{r_{0}}=\{x\in A\,:\,\textrm{dist}(x,I)> r_{0}\}$ and denoting
$\mathcal{U}^{r_{0}/8}=\{x\in G\,:\, \textrm{dist}(x,A^{r_{0}})\leq r_{0}/8\}$,
we may apply the local  bound estimates \cite{La-So-Ur} obtaining,
\begin{equation}
\left\Vert w_{t}\right\Vert_{L^{\infty }(\mathcal{U}^{r_{0}/8}\times(0,T))}
\leq C\left\Vert g\right\Vert_{H^{1/2,1/4}(A\times(0,T))}
\leq C\left\Vert g\right\Vert _{C^{1,1}(A\times (0,T))},
\label{101}
\end{equation}
\begin{equation}
\left\Vert w_{tt}\right\Vert _{L^{\infty }(\mathcal{U}^{r_{0}/8}\times(0,T))}
\leq C\left\Vert g\right\Vert _{H^{1/2,1/4}(A\times(0,T))}
\leq C\left\Vert g\right\Vert _{C^{1,1}(A\times (0,T))},
\label{102}
\end{equation}
where $C$ depends on $r_{0}^2T^{-1},L$.
We may also think to $w(\cdot ,t)$, for a fixed $t\in(0,T)$,
 as the solution of the elliptic problem
\begin{equation*}
\left\{\begin{array}{ll}
\Delta w(x,t)=w_{t}(x,t) & \textrm{ in }G, \\[2mm]
\frac{\partial w}{\partial \nu }(x,t)=0, & \textrm{ on }A,
\end{array}\right.
\end{equation*}
and, similarly, we may think to $w_{t}(\cdot ,t)$ as the solution of the elliptic problem
\begin{equation*}
\left\{\begin{array}{ll}
\Delta w_{t}(x,t)=w_{tt}(x,t) & \textrm{ in }G, \\[2mm]
\frac{\partial w_{t}}{\partial \nu }(x,t)=0,& \textrm{ on }A.
\end{array}\right.
\end{equation*}
By $L^{p}$ regularity estimates (see \cite{Gi-Tr}), by (\ref{101}), (\ref{102}),
by trace inequalities and by the immersion of $W_{p}^{2-1/p}$ in $H^{2-1/p}$ for $p>2,$ we have
\begin{equation*}
\underset{t\in [0,T]}\sup(\left\Vert w(\cdot ,t)\right\Vert_{H^{2-1/p}(\Sigma)}
+\left\Vert w_{t}(\cdot ,t)\right\Vert_{H^{2-1/p}(\Sigma)})
\leq C\left\Vert g\right\Vert _{C^{1,1}(A\times(0,T))},
\end{equation*}
for any $p>2,$ where $C$ depends on $L,r^2_{0}/T$ only. Therefore
\begin{equation}
\left\Vert w\right\Vert_{H^{\alpha ,\alpha /2}(\Sigma \times (0,T))}\leq
C\left\Vert g\right\Vert_{C^{1,1}(A\times (0,T))},
\label{104}
\end{equation}
with $\alpha =2-1/p>3/2,$ where $C$ depends on $L$, $r^2_{0}/T$ only.
By interpolation we have
\begin{equation}
\left\Vert w\right\Vert_{H^{3/2,3/4}(\Sigma \times (0,T))}
\leq C\left\Vert w\right\Vert_{H^{\alpha ,\alpha /2}
(\Sigma \times (0,T))}^{1-\theta}\left\Vert w\right\Vert_{L^{2}(\Sigma\times(0,T))}^{\theta},
\label{105}
\end{equation}
where $\theta $ is given by $(1-\theta )\alpha =3/2$ (see \cite{Li-Ma}).
By (\ref{36}), (\ref{45}), (\ref{104}), (\ref{105}), choosing $p=4$ we have
\begin{equation}
\left\Vert w\right\Vert_{H^{3/2,3/4}(\Sigma \times (0,T))}\leq C\,\varepsilon^{1/7},
\label{108}
\end{equation}
where $C$ depends on $L,r^2_{0}/T,E$ only.
Let $P_{1}\in \Sigma $ be such that $\partial\Omega\cap B_{r_{0}}(P_{1})\subset\Sigma $.
By Theorem \ref{teo3.3.1} and (\ref{108}), we get
\begin{equation}
\left\Vert w(\cdot,\overline{t})\right\Vert _{L^{2}(B_{\overline{\theta }r_{0}}(P_{2}))}
\leq C\varepsilon ^{\overline\tau/7},
\label{103}
\end{equation}
where $P_{2}$, $\overline{\theta}$, $\overline\tau$ are as in the
above theorem, $C\geq 1$ depends on $L,r^2_{0}/T,t_{1},E$
only and $\overline{t}$ is the point in (\ref{89bis}).
Now, let $\sigma$ be an arc in $\tilde{V}_{r}$ joining $\overline{x}$ with $P_{2}$
(since $\overline{\theta }r_{0}>r,$ the point $P_{2}\in\tilde{V}_{r}).$
Let us define $\{x_{i}\}$, $i=1,2,...,s$, as follows: $x_{1}=P_{2}$,
$x_{i+1}=\sigma(\eta_{i})$, where
$\eta _{i}=\max\{\eta\,:\,\left\vert \sigma(\eta)-x_{i}\right\vert =2\overline{\theta }r\}$ if
$\left\vert x_{i}-\overline{x}\right\vert >2\overline{\theta }r$, otherwise $i=s$ and stop the process.
By construction, the balls $B_{\overline{\theta}r}(x_{i})$ are pairwise disjoint,
$\left\vert x_{i+1}-x_{i}\right\vert =2\overline{\theta}r,$ for $i=1,2,...,s-1$,
$\left\vert x_{s}-\overline{x}\right\vert \leq 2\overline{\theta }r$. Notice that $s\leq\frac{C}{r^{n}}$,
where $C$ depends on $M$ only.
By an iterated application of the two-spheres one-cylinder inequality
over the chain of balls $B_{r_{1}}(x_{i}),$ for $r_{1}=\frac{\overline{\theta }}{26\sqrt{2}}r$,
by \eqref{reg-lib}, \eqref{81}, \eqref{89bis}, \eqref{103}, \eqref{106},
we get, for $t_1\leq t\leq T$,
\begin{equation}
\int_{\Omega _{1}\setminus G}u_{1}^{2}(x,t)\lambda_{1}^{2}(x,t)dx
\leq Cr_0^n\left(\frac{r}{r_0}
+\left( \frac{r_0}{r}\right)^{\frac{4n+10}{3n+6}}\varepsilon^{\frac{2\overline{\tau}^{s+1}}{7(3n+6)}}\right),
\label{113}
\end{equation}
where $C$ depends on $r_0^2(T-t_1)^{-1},L,E$.
With a suitable choice of $r=r(\varepsilon ),$ by standard arguments we get the thesis.
\end{proof}

\begin{proof}[Proof of Proposition \ref{cauchy-data2}]
By the use of the divergence theorem over the Lipschitz domains $\Omega_1$ and $G$
and the same arguments based on the application of the Gronwall inequality developed in Proposition \ref{cauchy-data},
we have that
\begin{eqnarray}
\int_{\Omega_1\setminus \overline G}\lambda_1(x,t)^2dx \leq C \int_{\partial(\Omega_1\setminus \overline G)}
|\partial_\nu\lambda_1(x,\bar{t})|d\sigma.
\end{eqnarray}
where $C>0$ is a constant depending on the a priori data only.
Moreover we observe that
$\partial(\Omega_1\setminus\overline G)\subset (\partial\Omega_1\setminus A)
\cup(\partial\Omega_2\cap\partial G\setminus U_A^{\frac{r_0}{2}})$.

Since $\partial_\nu\lambda_1=0$ on $\partial\Omega_1\setminus A$ and since $\partial_\nu\lambda_2=0$
on $\partial\Omega_2\setminus A$ we found that
\begin{eqnarray*}
\int_{\Omega_1\setminus\overline G}\lambda_1(x,t)^2dx &\leq& C
\int_{(\partial\Omega_2\cap\partial G)\setminus U_A^{\frac{r_0}{2}}}|\partial_\nu\lambda_1(x,\bar{t})-
\partial_\nu\lambda_2(x,\bar{t}) |dx\leq \\
&\leq& C_1 r_0^{n+1}\max_{\partial G}|\nabla \lambda_1(x,\bar{t})-\nabla \lambda_2(x,\bar{t})|,
\end{eqnarray*}
where $C_1$ is a constant depending on the a-priori data only.
By the same argument of Proposition \ref{cauchy-data} and using the same notations we get
\begin{eqnarray*}
&&\int_{\Omega_1\setminus\overline G}\lambda_1(x,t)^2dx\leq
C_2r_0^n\left(\max_{\partial G}|w|+\max_{\partial G}|\tilde{w}|
+r_0\max_{\partial G}|\nabla w| +r_0\max_{\partial G}|\nabla\tilde{w}| \right),
\end{eqnarray*}
where $C_2$ is a constant depending on the a priori data only.

In order to control the maximum of $w, \tilde{w}$ and their gradients we argue as in Proposition 5.4 of \cite{Ca-Ro-Ve}.
We carry out our analysis for the term $\nabla w$, the other cases being analogous.
Let $P_1\in A$ be such that $\partial \Omega \cap B_{r_0}(P_1)\subset A$ and let
$P_2= P_1 -\tilde{\theta}r_0\nu$ with $0<\tilde{\theta}<\frac{1}{4}$ and where
$\nu$ denotes the outer unit normal to $\Omega_1$ at $P_1$.
Now by Theorem \ref{teo3.3.1} arguing as in \eqref{103},
we may infer
\begin{eqnarray}
\label{CauchyStabilityEstimate}
\|w(\cdot,\bar{t})\|_{L^2(B_{\tilde{\theta}r_0}(P_2))}\le \tilde{C}\varepsilon^{\gamma},
\end{eqnarray}
where $\tilde{C}>0, 0<\gamma<1$ are constants depending on the a-priori data and on $\tilde{\theta}$ only.

Given $z\in \mathbb{R}^n, \xi\in \mathbb{R}^n, |\xi|=1, \theta>0, r>0$, we shall denote by
\begin{eqnarray}
C(z,\xi,\theta,r)= \{x\in \mathbb{R}^n\ :\ \frac{(x-z)\cdot \xi}{|x-z|}>\cos(\theta), \ |x-z|<r \}
\end{eqnarray}
the intersection of the ball $B_r(z)$ and the open cone having vertex $z$, axis in the direction
$\xi$ and width $2\theta$. Since $\partial G$ is of Lipschitz class with constant $r_0, L$ for any
$z\in \partial G$ there exists $\xi\in \mathbb{R}^n$, $|\xi|=1$, such that $C(z,\xi, \theta,r_0)\subset G$,
where $\theta=\arctan\frac{1}{L}$.

Let $(\bar{x},\bar{t})\in \partial G$ be such that
$|\nabla w(\bar{x}, \bar{t})|=\|\nabla w(\cdot,\bar{t})\|_{L^{\infty}(\partial G)}$.
Now dealing as in Proposition 5.4. \cite{Ca-Ro-Ve},
we combine the inequality \eqref{CauchyStabilityEstimate}
with an iterated use of the two--sphere and one--cylinder inequality (Theorem \ref{tre-teo})
within the cone $C(\bar{x}, \xi, \theta, r_0)$ obtaining the following estimate
\begin{eqnarray}
\label{uniquecontinuationcone}
\|w(\cdot,\bar{t})\|^2_{L^2(B_{\rho_{k(r)}}(x_{k(r)}))}
\leq C\left(1+\frac{T^2}{\rho^4_{k(r)}} \right)^{1-\bar{\tau}^{k(r)-1}}
\varepsilon^{\beta_1\bar{\tau}^{k(r)-1}},
\end{eqnarray}
with
\begin{eqnarray*}
&&\rho_0= a_1\sin\theta,\qquad
\rho_k= \chi\rho_{k-1},\qquad
d_1= a_1(1-\sin\theta),\\
&&\frac{|\log\frac{r}{d_1}|}{|\log \chi|}\leq k(r)-1\le \frac{|\log\frac{r}{d_1}|}{|\log \chi|}+1,
\end{eqnarray*}
where  $0<\bar{\tau}<1, 0<\beta_1<1, 0<\chi<1, a_1>0$ are positive constants depending on the a-priori
data only and where $x_{k(r)}$
is a point lying on the axis $\xi$ of the cone $C(\bar{x}, \xi, \theta, r_0)$
at a distance $\chi^{k(r)-1}\cdot d_1 + \rho_{k(r)}$ from $\bar{x}$ with $0<r<d_1$.
By the interpolation inequality \eqref{106} stated in Proposition \ref{cauchy-data} and the definition of $\rho_{k(r)}$
we have that \eqref{uniquecontinuationcone} leads to
\begin{eqnarray}
\label{uniquecontinuationconegradient}
\|\nabla w(\cdot,\bar{t})\|_{L^{\infty}(B_{\rho_{k(r)}}(x_{k(r)}))}
\leq \frac{C}{r_0}\chi^{\beta_3(1-k(r))}\varepsilon^{\beta_2\bar{\tau}^{k(r)-1}},
\end{eqnarray}
where $C>0, 0<\beta_2<1, 0<\beta_3<1$ are constants depending on the a-priori data only.
We consider the point $x_r=\bar{x}+r\xi$. We have that $x_r\in B_{\rho_{k(r)}}(x_{k(r)})$.
From \eqref{uniquecontinuationconegradient} and from the $C^{1,\alpha}$ regularity of $w$ we have that
\begin{eqnarray}
|\nabla w (\bar{x}, \bar{t})|\leq\frac{C}{r_0} \left(\left(\frac{r}{r_0}\right)^{\alpha} + \chi^{\beta_3(1-k(r))}
\varepsilon^{\beta_2\bar{\tau}^{k(r)-1}} \right),
\end{eqnarray}
where $C>0$ depends on the a-priori data only. Minimizing with respect to $r$ we obtain the desired estimate.
\end{proof}
Let us consider now the proofs of Propositions \ref{lemma-lower} and \ref{stimadalbasso}.
For this purpose we need a Harnack inequality, its version at the boundary and a technical lemma
(see Lemma \ref{tech} below).

The first tool can be found in \cite{Mo}. We state a Harnack inequality at the boundary,
postponing its proof to the next Section \ref{sec5}. Let us remark here that the thesis
holds true weakening the regularity assumptions on the boundary ($C^{0,1}$
instead of $C^{1,\alpha}$), on $\gamma$
and considering operators of more general form such as
$\mathrm{div (a(x,t)\nabla u)-u_t}$, where $a$ is bounded and satisfies a uniformly ellipticity condition.

\begin{proposition}[Harnack inequality at the boundary]
\label{harnack-lemma}
Let $\Omega$ be a bounded domain with $C^{1,\alpha}$ boundary with constants $r_0, L$.
Let $\gamma$ satisfying \eqref{gamma} be such that there exists a positive constant $\overline \gamma$
such that $\gamma\leq\overline\gamma$.
Let us denote by $T_1,T_2$ two numbers in the time interval $[0,T]$.
Let $u\in C^{1,\alpha}(\overline\Omega\times[0,T])$ be a positive solution to
$$\left\{\begin{array}{ll}
u_t-\Delta u=0 & \textrm{in }\Omega\times[T_1,T_2]\\[2mm]
\frac{\partial u}{\partial\nu}+\gamma u=0 & \textrm{on }\Gamma\times[T_1,T_2],
\end{array}\right.$$
where $\Gamma$ is an open portion compactly contained in $\partial\Omega$.
Assume $T_1<t_1<t_2<t_3<t_4\leq T_2$.
Then for $\rho<\rho_0$ there exists a positive constant $C$ depending on $\rho_0,\rho,t_1,t_2,t_3,t_4$
such that
\begin{equation}
\label{harnack}
\sup_{(B_\rho\cap\Omega)\times[t_1,t_2]}u\leq C\inf_{(B_\rho\cap\Omega)\times[t_3,t_4]}u.
\end{equation}
\end{proposition}
In order to state next result, let us introduce the following notation.
We shall denote by $b_0,b_1$ two positive constants such that
$$b_0\leq u(x,t)\leq b_1\qquad \forall \ (x,t)\in\overline\Omega\times[t_1,T],$$
(by \eqref{(1)-1} we can take $b_0=c_0\Phi_1$,
whereas the existence of $b_1$ is guaranteed by \eqref{phi001} and \eqref{reg-lib}).

\begin{lemma}
\label{tech}
Let the hypothesis of Theorem \ref{main-theor} be satisfied. We have that
\begin{eqnarray}\label{lowerNeumann}
&&\int_{\Omega}\lambda^2(x,t)dx + b_0^2\int_{t_1}^t\int_{\Omega}|\nabla \lambda(x,t)|^2dxdt\nonumber\\
&\leq& C_0b_0^2(1+ e^{b_0^2tr_0^{-2}})
\int_{t_1}^t\int_{\partial\Omega}|u^2(x,\tau)\partial_{\nu}\lambda(x,\tau)|^2
d\sigma d\tau, \quad \forall\,t\in [t_1,T]
\end{eqnarray}
where $C_0>0$ is a constant depending on the a-priori data only.
\end{lemma}
\begin{proof}
For a sake of brevity, we shall denote along the proof $h=u^2\partial_{\nu}\lambda$.
By the weak formulation of problem \eqref{eq-lambda} we obtain that
\begin{eqnarray}
&&\frac{1}{2}\int_{t_1}^t\int_{\Omega}(\lambda^2)_{\tau}dxd\tau+
\int_{t_1}^{t}\int_{\Omega}u^2|\nabla\lambda|^2dx d\tau\nonumber\\
&=&\int_{t_1}^t\int_{\partial \Omega}h\lambda d\sigma d\tau.
\end{eqnarray}
By the H\"{o}lder inequality, we get
\begin{eqnarray}
\label{parts}
&&\frac{1}{2}\int_{\Omega}\lambda^2dx+
\int_{t_1}^{t}\int_{\Omega}u^2|\nabla\lambda|^2dxd\tau\nonumber\\
&\leq&
\left(\int_{t_1}^t\int_{\partial \Omega}h^2 d\sigma d\tau\right)^{\frac{1}{2}}
\left(\int_{t_1}^t\int_{\partial \Omega}\lambda^2 d\sigma d\tau\right)^{\frac{1}{2}}.
\end{eqnarray}
We recall the following trace inequality (see \cite{Li})
\begin{eqnarray}
\label{trace}
\|\lambda(\cdot,t)\|_{L^2(\partial \Omega)}\leq
C \left(r_0^{-1/2}\|\lambda(\cdot,t)\|_{L^2(\Omega)}
+r_0^{1/2}\|\nabla \lambda(\cdot, t)\|_{L^2(\Omega)}\right)
\end{eqnarray}
where $C>0$ is a constant depending on $r_0$ and $L$ only.
Raising to the square the latter and integrating over the interval $[t_1,t] $ we have that
\begin{eqnarray*}
\int_{t_1}^t\int_{\partial\Omega}\lambda^2d\sigma d\tau
\leq C \left(r_0^{-1}\int_{t_1}^t\int_{\Omega}\lambda^2dxd\tau
+r_0\int_{t_1}^t\int_{\Omega}|\nabla\lambda|^2dxd\tau \right).
\end{eqnarray*}
Plugging the above estimate in the right hand side of \eqref{parts} and using the Young inequality we get
\begin{eqnarray*}
&&\frac{1}{2}\int_{\Omega}\lambda^2dx + b_0^2\int_{t_1}^t\int_{\Omega}|\nabla \lambda|^2dxd\tau\nonumber\\
&\leq&\frac{Cr_0}{\delta}\int_{t_1}^t\int_{\partial \Omega}h^2d\sigma d\tau
+\delta r_0^{-2}\int_{t_1}^{t}\int_\Omega\lambda^2dx d\tau
+\delta\int_{t_1}^t\int_{\Omega}|\nabla\lambda|^2dx d\tau, \quad \forall \ t\in [t_1,T].
 \end{eqnarray*}
Choosing $\delta=\frac{1}{2}b_0^2$ we obtain that
\begin{eqnarray}
\label{intermediate}
&&\int_{\Omega}\lambda^2dx\\
&\leq&
Cr_0b_0^2\int_{t_1}^t\int_{\partial \Omega}h^2d\sigma d\tau
+ b_0^2r_0^{-2}\int_{t_1}^{t}\int_\Omega\lambda^2dx d\tau, \quad\forall\,t\in[t_1,T].\nonumber
\end{eqnarray}
Moreover, by the Gronwall inequality we infer that
\begin{eqnarray}
\label{gronwall}
\int_{\Omega}\lambda^2dx \leq
Cr_0b_0^2 e^{b_0^2tr_0^{-2}}\int_{t_1}^t\int_{\partial \Omega}h^2d\sigma d\tau.
\end{eqnarray}
Finally, integrating \eqref{gronwall} over the interval $[t_1,t]$ and combining the obtained inequality
with \eqref{intermediate} we get the desired estimate \eqref{lowerNeumann}.
\end{proof}

\begin{proof}[Proof of Proposition \ref{lemma-lower}]
Let $(x_0,t_0)\in\overline\Omega\times[t_1,T]$ be such that
$$u(x_0,t_0)=\min_{(x,t)\in\overline\Omega\times[t_1,T]}u(x,t).$$
By maximum principle
$(x_0,t_0)\in(\overline I\times[t_1,T])\cup(\overline A\times[t_1,T])\cup(\Omega\times\{t_1\})$.
Let us consider separately the three pieces of the boundary.

i) $(x_0,t_0)\in \bar A\times[t_1,T]$.
Being $\partial_\nu u(\cdot,t_0)\in C^\alpha(\partial\Omega)$ and
$\partial_\nu u(\cdot,t_0)=g(\cdot,t_0)$ on $A$, we have that $\partial_\nu u(\cdot,t_0)\geq0$
on $\bar A$, that contradicts Hopf maximum principle. Thus we can conclude that $(x_0,t_0)$
doesn't belong to $\bar A\times[t_1,T]$.

ii) $(x_0,t_0)\in \overline I\times[t_1,T]$.
First we fix $y_0\in A$ and $y_1=y_0-\frac{r_0}{2}\nu(y_0)$.
Without loss of generality, assume $t_0>t_1$. We divide the interval
$\left[\frac{t_0+t_1}{2},t_0\right]$ into $N$ subintervals $[\tilde t_i,\tilde t_{i+1}]$,
$i=1,\dots,N$, where $\tilde t_N=\frac{t_0+t_1}{2}$ and $\tilde t_0=t_0$.
We shall quantify $N$ later on. By Harnack inequality at the boundary (Proposition \ref{harnack-lemma}) we have
\begin{equation}
\label{1)-5}
\inf_{B_\rho(x_0)\cap\Omega}u(x,t_0)\geq C \sup_{B_\rho(x_0)\cap\Omega}u(x,\tilde t_1),
\end{equation}
where $\rho$ is such that $\rho\leq 4 r_0$. There exists $x_1\in B_\rho(x_0)\cap\Omega$ such that
$B_{\rho/8}(x_1)\subset B_\rho(x_0)\cap\Omega$. Thus
$$\sup_{B_\rho(x_0)\cap\Omega}u(x,\tilde t_1)\geq \sup_{B_{\rho/8}(x_1)}u(x,\tilde t_1).$$
We denote be $\sigma$ a continuous path joining $y_1$ and $x_1$ and define $x_i$, $i=1,\dots,N$, as follows
$x_{i+1}=\sigma(s_i)$, where $s_i=\max\{s\,:\,|\sigma(s)-x_i|=\rho/8\}$ if $|y_1-x_i|>\rho/8$,
otherwise $i=N$ and stop the process.
Trivially
$$\sup_{B_{\rho/8}(x_1)}u(x,\tilde t_1)\geq \inf_{B_{3\rho/4}(x_1)}u(x,\tilde t_1).$$
By Harnack inequality in the interior \cite{Mo} we have
$$\inf_{B_{3\rho/4}(x_1)}u(x,\tilde t_1)\geq C\sup_{B_{3\rho/4}(x_1)}u(x,\tilde t_2).$$
Then
$$\sup_{B_{3\rho/4}(x_1)}u(x,\tilde t_2)\geq \sup_{B_{\rho/8}(x_2)}u(x,\tilde t_2).$$
Summarizing we get
\begin{equation}
\label{1)-6}
\sup_{B_{\rho/8}(x_1)}u(x,\tilde t_1)\geq C\sup_{B_{\rho/8}(x_2)}u(x,\tilde t_2),
\end{equation}
where $C$ depends on $t_0,t_1$.

Iterating this process along a chain of balls $\{B_{\rho/8}(x_i)\}_{i=1}^N$, we obtain the estimate
\begin{equation}
\label{1)-9}
\sup_{B_{\rho/8}(x_1)}u(x,\tilde t_1)\geq C^N\sup_{B_{\rho/8}(x_N)}u(x,\tilde t_N)
\geq C^N u(y_1,\tilde t_N).
\end{equation}
By Taylor formula, recalling that $u>0$ and by \eqref{reg-lib}, \eqref{039}, we have that
\begin{eqnarray*}
u(y_1,\tilde t_N)&\geq& g(y_0,\tilde t_N)\frac{\rho}{8}-
C r_0^{-\alpha}\|g\|_{C^{0,\alpha}(A\times[0,T])}\left(\frac{\rho}{8}\right)^{1+\alpha}\\[2mm]
&\geq& \frac{\rho}{8r_0}\left(\Phi_1-CE\left(\frac{\rho}{8r_0}\right)^\alpha\right).
\end{eqnarray*}
Now choosing $\rho=\min\left\{4r_0,8r_0\left(\frac{\Phi_1}{2EC}\right)^{1/\alpha}\right\}$,
we get the thesis.

\noindent iii) $(x_0,t_0)\in\overline\Omega\times\{t_1\}$.
This case can be treated similarly as ii).
\end{proof}

\begin{proof}[Proof of Proposition \ref{stimadalbasso}]
In the sequel we shall maintain the notation $h=u^2\partial_{\nu}\lambda$.
After straightforward computation we observe that
\begin{eqnarray*}
\frac{\tilde{g}}{g}-1= \frac{u}{g}\partial_{\nu}\lambda+\lambda,
\end{eqnarray*}
from this identity and by \eqref{reg-lib} we get
\begin{eqnarray}
\label{int}
\left\|\frac{\tilde{g}}{g}-1\right \|_{L^2(A\times [t_1,T])}
\leq \frac{1}{C_0b_1}\|h\|_{L^2(A\times [t_1,T])}
+\|\lambda \|_{L^2(\partial \Omega\times[t_1,T])}
\end{eqnarray}
where $C_0$ has been introduced in \eqref{reg-lib}.
Now, by integrating the trace estimate \eqref{trace} over the time interval $[t_1,T]$ and by \eqref{lowerNeumann} we get
\begin{eqnarray}
\label{comb}
&&\int_{t_1}^T\int_{\partial\Omega}\lambda^2d\sigma d\tau\leq C\left(r_0^{-1}\int_{t_1}^T\int_{\Omega}\lambda^2dxd\tau+
r_0\int_{t_1}^t\int_{\Omega}|\nabla\lambda|^2dxd\tau \right)\nonumber\\
&\leq&CC_0(1+e^{b_0^2r_0^{-2}t})(1+ b_0^2r_0^{-2}t)\int_{t_1}^T\int_{A}h^2d\sigma d\tau.
\end{eqnarray}
At this stage we use the above inequality to control the right hand side of \eqref{int} obtaining the following
\begin{eqnarray}
\label{first}
\left\|\frac{\tilde{g}}{g}-1\right\|_{L^2(A\times [t_1,T])}
\leq C_1 \|h \|_{L^2(A\times[t_1,T])},
\end{eqnarray}
where $C_1=\left( \frac{1}{C_0b_1}+[CC_0(1+e^{b_0^2r_0^{-2}T})(1+ b_0^2r_0^{-2}T)]^{\frac{1}{2}}\right)$.
Recalling that for every $c\in \mathbb{R}$ we have
\begin{eqnarray*}
\left\|\frac{\tilde{g}}{g}-c\right\|_{L^2(A\times [t_1,T])}
\geq\left\|\frac{\tilde{g}}{g}-\left(\frac{\tilde{g}}{g}\right)_{A\times [0,T]}\right \|_{L^2(A\times [0,T])}
\geq\Phi_0,
\end{eqnarray*}
from \eqref{first} we infer that
\begin{eqnarray*}
\left\|\frac{\tilde{g}}{g}-\left(\frac{\tilde{g}}{g}\right)_{A\times [0,T]}\right\|_{L^2(A\times [0,T])}
\leq C_1\|h\|_{L^2(A\times[t_1,T])}.
\end{eqnarray*}
At this point we claim that for any $\rho>0$ and for any $x_0\in\Omega_{\rho}$ it holds
\begin{eqnarray*}
\|h\|^2_{L^2(A\times[t_1,T])}
\leq \frac{C}{r_0^{n+2}}\int_{t_1}^{T}\int_{B_{\rho}(x_0)}\lambda^2 dxd\tau,
\end{eqnarray*}
where $C>0$ is a constant depending on the a-priori data and on $\rho$ only.
Our claim, now, follows by standard arguments based on Theorem \ref{tre-teo} and the corresponding version
in the interior (see \cite[Proposition 4.1.3]{Ve2}).
\end{proof}

\section{Proof of Proposition \ref{harnack-lemma} (Harnack Inequality at the Boundary)}
\label{sec5}
The proof of Proposition \ref{harnack-lemma} can be obtained as in \cite{Mo}, where
the result relies on two Lemmas labeled as Lemma 1 and Lemma 2. For the sake of completeness, we state and sketch the proof
of them in the present situation in the two following lemmas.
For $r>0$, we shall denote by $S(r)$ the cylinder $|t|<r^2$, $|x|<r$, by $S^-(r)$
the cylinder $0<-t<r^2$, $|x|<r$ and by $S^+(r)$
the cylinder $0<t<r^2$, $|x|<r$.
We denote also $B_r^+=B_r\cap\{x_n>0\}$.
\begin{lemma}
\label{lemma1}
Let $u(x,t)>0$ be a solution of the problem
\begin{equation}
\label{33}
\left\{\begin{array}{ll}
\mathrm{div}(\sigma (x)\nabla u(x,t))=u_{t}(x,t), & x\in B^{+}_2,|t|\leq\overline t,\,\overline t>1,\\[2mm]
\sigma (x)\nabla u(x,t)\cdot \nu (x)+\gamma_0 (x,t)\text{ }u(x,t)=0, &
\textrm{for }\left\vert x\right\vert <2,\,x_{n}=0,|t|\leq\overline t,
\end{array}\right.
\end{equation}
where
\begin{equation}
\label{33b}
\lambda \left\vert \xi \right\vert ^{2}\leq \sigma (x)\xi \cdot \xi \leq
\Lambda \left\vert \xi \right\vert ^{2},
\end{equation}
for every $x\in B^{+}_2$, $\xi \in \mathbb{R}^{n}$ and
$\left\vert \gamma_0(x,t)\right\vert \leq \overline{\gamma}_0$.
Let $\frac{1}{2}\leq \rho <r\leq1 $ and $\mu =\Lambda +\frac{1}{\lambda }$.
Then there exists a constant $C_4=C_{4}(n,\overline{\gamma}_0)$ such that
\begin{eqnarray*}
&&\underset{S(\rho )}{\sup}\,u^{p}\leq \frac{C_{4}}{(r-\rho )^{n+2}}
\iint_{S(r)}u^{p}dxdt,\textrm{ for every }0<p<\mu ^{-1},\\[2mm]
&&\underset{S^{-}(\rho )}{\sup }u^{p}\leq \frac{C_{4}}{(r-\rho )^{n+2}}
\iint_{S^{-}(r)}u^{p}dxdt,\textrm{ for every } -\mu^{-1}<p<0.
\end{eqnarray*}
\end{lemma}
\begin{proof}
Let $\Phi(x,t)$ be a test function such that
\begin{equation*}
\left\{\begin{array}{l}
\Phi \in C^{1}(S(1)),\\
\Phi =0,\text{ for }(x,t)\in\partial B_1\times[-1,1].
\end{array}\right.
\end{equation*}
By \eqref{33}, integrating over
$B^{+}_1\times (t_{1},t_{2})$, $t_1,t_2\in(-1,1)$, and taking into
account the Robin condition, we get
\begin{eqnarray*}
&&\int_{t_{1}}^{t_{2}}\int_{B^{+}_1}u_{t}\Phi
+\int_{t_{1}}^{t_{2}}\int_{B^{+}_1}\sigma \nabla u\cdot\nabla\Phi\\
&=&\int_{t_{1}}^{t_{2}}\int_{B^{+}_1}\sigma \nabla u\cdot\nu\Phi=
-\int_{t_{1}}^{t_{2}}\int_{I(1)}\gamma_0(x,t)u\Phi,
\end{eqnarray*}
where $I(1)=\{(x\in\mathbb R^n\,:\,x=(x',0),\,|x|\leq1\}$.
Let us set $v=u^{\frac{p}{2}}$, $\Phi=u^{p-1}\psi ^{2}$.
Then we obtain,
\begin{eqnarray*}
&&\int_{t_{1}}^{t_{2}}\int_{B^{+}_1}\frac{2}{p}vv_{t}\psi^{2}
+\int_{t_{1}}^{t_{2}}\int_{B^{+}_1}\frac{4(p-1)}{p^{2}}\sigma \nabla v\cdot\nabla v\psi ^{2}\\
&&+\int_{t_{1}}^{t_{2}}\int_{B^{+}_1}\frac{4}{p}v\psi \sigma\nabla v\cdot\nabla \psi=
-\int_{t_{1}}^{t_{2}}\int_{I(1)}\gamma_0 v^{2}\psi^{2}.
\end{eqnarray*}
Multiplying for $\frac{p}{4}$ and adding to both sides the term
$\int_{t_{1}}^{t_{2}}\int_{B^{+}_1}\frac{1}{4}(\dfrac{d}{dt}\psi ^{2})v^{2},$ we get
\begin{eqnarray}
\label{34}
&&\frac{1}{4}\int_{t_{1}}^{t_{2}}\int_{B^{+}_1}\frac{d}{dt}(v^{2}\psi^{2})+(1-\frac{1}{p})\int_{t_{1}}^{t_{2}}
\int_{B^{+}_1}\sigma \nabla v\cdot\nabla v\psi ^{2}\\
&=&-\int_{t_{1}}^{t_{2}}\int_{B^{+}_1}v\psi\sigma \nabla v\cdot\nabla\psi+
\frac{1}{2}\int_{t_{1}}^{t_{2}}\int_{B^{+}_1}\psi \psi _{t}v^{2}-\frac{p}{4}
\int_{t_{1}}^{t_{2}}\int_{I(1)}\gamma_0 v^{2}\psi ^{2}.\nonumber
\end{eqnarray}
The term $\int_{t_{1}}^{t_{2}}\int_{I(1)}\gamma_0 v^{2}\psi ^{2}$
can be estimated by trace theorem (\cite[Th. 5.22]{Ad}) and Poincar\'{e} inequality, that is
\begin{equation}
\label{35}
\left\vert \int_{t_{1}}^{t_{2}}\int_{I(1)}\gamma_0 v^{2}\psi ^{2}\right\vert
\leq C\overline{\gamma }_0\int_{t_{1}}^{t_{2}}\int_{B^{+}_1}\left\vert \nabla(v^{2}\psi ^{2})\right\vert,
\end{equation}
where C is an absolute constant.
Then, by spreading the gradient of the right hand side of \eqref{35}
and by applying the inequality $2ab\leq a^2+b^2$, we get
\begin{eqnarray}
\label{36}
&&\left\vert\frac{p}{4}\int_{t_{1}}^{t_{2}}\int_{I(1)}\gamma_0 v^{2}\psi^{2}\right\vert
\leq C\overline{\gamma}_0\frac{\left\vert p\right\vert}{4}
\int_{t_{1}}^{t_{2}}\int_{B^{+}_1}\left\vert \nabla (v^{2}\psi^{2})\right\vert\\
&\leq& \tilde\varepsilon\int_{t_{1}}^{t_{2}}\int_{B^{+}_1}\left\vert \nabla v\right\vert ^{2}\psi ^{2}+
+\frac{1}{\tilde\varepsilon}\frac{C^{2}\overline{\gamma}_0^{2}
\left\vert p\right\vert ^{2}}{16}\int_{t_{1}}^{t_{2}}\int_{B^{+}_1}v^{2}\psi ^{2}\nonumber\\
&&+\frac{C\overline{\gamma}_0\left\vert p\right\vert}{4}
\int_{t_{1}}^{t_{2}}\int_{B^{+}_1}v^{2}(\psi^{2}+\left\vert \nabla \psi \right\vert ^{2}).\nonumber
\end{eqnarray}
We first consider the case $(1-\frac{1}{p})>0$.
By \eqref{33b}, \eqref{34}, \eqref{36} and by
\begin{equation}
\label{37}
\left\vert v\psi\sigma \nabla v\cdot\nabla \psi\right\vert
\leq \frac{1}{4\tilde\varepsilon}v^{2}\nabla\psi\cdot\sigma\nabla\psi
+\tilde\varepsilon\psi^{2}\nabla v\cdot\sigma\nabla v,
\end{equation}
we obtain, taking $\tilde\varepsilon=\frac{\lambda}{\lambda +1}\frac{1}{2}(1-\frac{1}{p})$,
\begin{eqnarray}
\label{39}
&&\frac{1}{4}\int_{t_{1}}^{t_{2}}\int_{B^{+}_1}\frac{d}{dt}(v^{2}\psi ^{2})
+\frac{\lambda }{2}(1-\frac{1}{p})\int_{t_{1}}^{t_{2}}\int_{B^{+}_1}\left\vert \nabla v\right\vert ^{2}\psi ^{2}\\
&\leq&\frac{1}{2}\frac{\left(\lambda +1\right)\Lambda}{\lambda}\frac{p}{p-1}
\int_{t_{1}}^{t_{2}}\int_{B^{+}_1}v^{2}\left\vert\nabla\psi\right\vert^{2}
+\frac{1}{2}\int_{t_{1}}^{t_{2}}\int_{B^{+}_1}v^{2}\left\vert\psi \psi _{t}\right\vert\nonumber\\
&&+C(\lambda ,\overline{\gamma}_0)\left(\frac{p^{3}}{p-1}+\left\vert p\right\vert\right)
\int_{t_{1}}^{t_{2}}\int_{B^{+}_1}v^{2}(\psi ^{2}+\left\vert \nabla \psi\right\vert ^{2})\nonumber.
\end{eqnarray}
In the case $(1-\frac{1}{p})<0$, we multiply \eqref{34} by $-1$.
By \eqref{33b}, \eqref{36}, \eqref{37},  choosing $\tilde\varepsilon=\frac{\lambda}{\lambda +1}\frac{1}{2}(\frac{1}{p}-1)$,
we obtain
\begin{eqnarray}
\label{40}
&&-\frac{1}{4}\int_{t_{1}}^{t_{2}}\int_{B^{+}_1}\dfrac{d}{dt}(v^{2}\psi ^{2})
+\frac{\lambda }{2}(\frac{1}{p}-1)\int_{t_{1}}^{t_{2}}\int_{B^{+}_1}\left\vert \nabla v\right\vert ^{2}\psi ^{2}\\
&\leq&\frac{1}{2}\frac{\left( \lambda +1\right)\Lambda }{\lambda }\frac{p}{1-p}
\int_{t_{1}}^{t_{2}}\int_{B^{+}_1}v^{2}\left\vert \nabla \psi \right\vert
^{2}+\frac{1}{2}\int_{t_{1}}^{t_{2}}\int_{B^{+}_1}v^{2}\left\vert\psi \psi _{t}\right\vert\nonumber\\
&&+C(\lambda ,\overline{\gamma}_0)\left(\frac{p^{3}}{1-p}+\left\vert p\right\vert\right)
\int_{t_{1}}^{t_{2}}\int_{B^{+}_1}v^{2}(\psi ^{2}+\left\vert\nabla\psi\right\vert^{2}).\nonumber
\end{eqnarray}
Inequalities \eqref{39}, \eqref{40} are analogous to the one in \cite[page 737]{Mo},
with the addition of the term
$C(\lambda ,\overline{\gamma}_0)(p^2\left|\frac{p}{1-p}\right|+\left\vert p\right\vert)
\int_{t_{1}}^{t_{2}}\int_{B^{+}_1}v^{2}(\psi^{2}+\left\vert \nabla \psi \right\vert ^{2})$.
Proceeding in the same way, we get the thesis.
\end{proof}

\begin{lemma}
\label{lemma2}
Let the hypothesis of Lemma \ref{lemma1} be fulfilled.
Then there exist constants $a$, $C_{5}$ such that
\begin{equation*}
|\{(x,t)\in S^{+}(1):\log u<-s+a\}|+
|\{(x,t)\in S^{-}(1):\log u>s+a\}|\leq \frac{C_{5}}{s},
\end{equation*}
for every $s>0$, where $C_{5}$ depends on $\lambda,\Lambda,n,\overline{\gamma}_0$,
and $a$ depends on $u$.
\end{lemma}
\begin{proof}
We consider the function $v=-\log u$
that solves the problem
\begin{equation}
\label{42}
\left\{\begin{array}{ll}
v_{t}-\mathrm{div}(\sigma \nabla v)=-\nabla v\cdot\sigma \nabla v, & x\in B^{+}_2,|t|<1, \\
\sigma \nabla v\cdot \nu=\gamma_0, & \left\vert x\right\vert <2,x_{n}=0,|t|<1.
\end{array}\right.
\end{equation}
Let $\psi^{2}(x)$ be a test function independent on $t$ and such that
$\psi(x)\geq 0$, $\psi(x)=0$ for $\left\vert x\right\vert =2$, $x_{n}>0$.
By \eqref{42}, we get, adding the null term $v\frac{d}{dt}\psi ^{2}$,
\begin{eqnarray}
\label{43}
&&\int_{t_{1}}^{t_{2}}\int_{B^{+}_2}\dfrac{d}{dt}(v\psi^{2})
+2\int_{t_{1}}^{t_{2}}\int_{B^{+}_2}\sigma \nabla v\cdot\nabla\psi\psi\nonumber\\
&&-\int_{t_{1}}^{t_{2}}\int_{I(2)}\gamma \psi^{2}
+\int_{t_{1}}^{t_{2}}\int_{B^{+}_2}\sigma\nabla v\cdot\nabla v\psi^{2}=0,
\end{eqnarray}
where $t_1,t_2\in (-1,1)$. By Schwarz inequality, since
\begin{equation*}
\int_{t_{1}}^{t_{2}}\int_{B^{+}_2}\sigma \nabla v\cdot\nabla \psi\psi
\geq-\frac{1}{4}\int_{t_{1}}^{t_{2}}\int_{B^{+}_2}\sigma \nabla v\cdot\nabla v\psi^{2}
-\int_{t_{1}}^{t_{2}}\int_{B^{+}_2}\sigma|\nabla\psi|^2,
\end{equation*}
and by \eqref{33b}, \eqref{43}, we get
\begin{eqnarray*}
&&\int_{B^{+}_2}v\psi ^{2}\left\vert _{t_{1}}^{t_{2}}\right. +\frac{\lambda}{2}
\int_{t_{1}}^{t_{2}}\int_{B^{+}_2}\left\vert \nabla v\right\vert^{2}\psi^{2}
\leq 2\int_{t_{1}}^{t_{2}}\int_{B^{+}_2}\sigma \nabla\psi\cdot\nabla\psi
+\int_{t_{1}}^{t_{2}}\int_{I(2)}\gamma_0 \psi ^{2}.
\end{eqnarray*}
Again by trace theorem (\cite[Th. 5.22]{Ad}) and Poincar\'{e} inequality we finally obtain
\begin{eqnarray}
\label{45}
&&\int_{B^{+}_2}v\psi ^{2}|_{t_{1}}^{t_{2}}+
\frac{\lambda }{2}\int_{t_{1}}^{t_{2}}\int_{B^{+}_2}\left\vert\nabla v\right\vert ^{2}
\psi^{2}\leq 2\int_{t_{1}}^{t_{2}}\int_{B^{+}_2}\sigma \nabla \psi\cdot\nabla\psi\nonumber\\
&&+2\overline{\gamma}_0\int_{t_{1}}^{t_{2}}\int_{B^{+}_2}|\psi\nabla \psi|
\leq C(\overline{\gamma}_0,\Lambda)\int_{t_{1}}^{t_{2}}\int_{B^{+}_2}(\psi ^{2}+\left\vert \nabla \psi\right\vert ^{2}).
\end{eqnarray}
Inequality \eqref{45} is analogous to the inequality in \cite{Mo2}, page 121, with the
addition of the term
$C(\overline{\gamma}_0,\Lambda)\int_{t_{1}}^{t_{2}}\int_{B^{+}_2}(\psi ^{2}+\left\vert \nabla \psi\right\vert ^{2})$.
Proceeding in the same way, we get the thesis.
\end{proof}
\noindent Let $P\in\Gamma$.
Owing to the boundary regularity of $\Omega$, there exists a rigid
transformation of coordinates such that $P\equiv0$ and
$$\Omega\cap B_{r_0}=\{x\in\mathbb R^n\,:\, x_n>\varphi(x')\},$$
where $\varphi\in C^{1,\alpha}(B'_{r_0})$ satisfies
$$\varphi (0)=\left\vert \nabla \varphi (0)\right\vert=0,\qquad
\left\Vert\varphi\right\Vert_{C^{1,\alpha}(B'_{r_0})}\leq Lr_{0}.$$
Defining the map $\Psi$ on $B_{r_2}$ as
$$\Psi(y)=(y',\varphi(y')+y_n),$$
we have that $\Psi\in C^{1,\alpha}(B_{r_2})$ and there exist $C_1,\theta_1,\theta_2$,
positive constants, $0<\theta_i<1$, $i=1,2$, depending on $L$ only such that for $r_i=\theta_i r_0$, $i=1,2$,
\begin{eqnarray*}
&&\Psi (B_{r_2})\subset B_{r_1},\\
&&\Psi (y',0)=(y',\varphi(y')),\textrm{ for every }y'\in B'_{r_2}\subseteq \mathbb{R}^{n-1},\\
&&\Psi (B^{+}_{r_2})\subset \Omega \cap B_{r_1},\\
&&\frac{1}{2}\left\vert y-z\right\vert \leq \left\vert \Psi(y)-\Psi(z)\right\vert
\leq C_{1}\left\vert y-z\right\vert,\forall \ y,z\in B'_{r_2},\\
&&\det D\Psi=1.
\end{eqnarray*}
Denoting by
\begin{eqnarray*}
&&\sigma(y)=(D\Psi^{-1})(\Psi (y))(D\Psi^{-1})^{T}(\Psi(y)),\\
&&\tilde u(y,t)=u(\Psi (y),t),\qquad
\gamma'(y,t)=\gamma(\Psi (y),t),
\end{eqnarray*}
we have
\begin{eqnarray*}
&&\frac{1}{2^{n+2}}\left\vert \xi \right\vert ^{2}\leq \sigma (y)\xi \cdot \xi
\leq C_{2}\left\vert \xi \right\vert ^{2},\quad\forall \ y\in B^{+}_{\rho_{2}},\xi \in \mathbb{R}^{n},\\
&&\left\vert \sigma(y)-\sigma(z)\right\vert \leq\frac{C_{3}}{\rho _{0}}
\left\vert y-z\right\vert,\quad\forall \ y,z\in B^{+}_{\rho _{2}},
\end{eqnarray*}
where $C_{2}$, $C_{3}$ depend on $L$. Moreover $\tilde u(y,t)$ satisfies the problem
\begin{equation}
\label{til}
\left\{\begin{array}{ll}
\mathrm{div}(\sigma(y)\nabla \tilde u(y,t))=\tilde u_{t}(y,t), & B^{+}_{\rho _{2}}\times(T_{1},T_{2}),\\[2mm]
\sigma (y',0)\nabla \tilde u((y',0),t)\cdot \nu(y',0)+\gamma'((y',0),t)\tilde u((y',0),t)=0, &
\left\vert y'\right\vert <\rho_{2},\,t\in(T_{1},T_{2}),
\end{array}\right.
\end{equation}
where $\nu (y',0)=(0,0,...,-1)$.
By a standard scaling argument and applying Lemma \ref{lemma1} and \ref{lemma2},
we obtain Proposition \ref{harnack-lemma}.

\end{document}